\documentclass[a4paper,10pt]{article}
\usepackage[english]{babel}
\usepackage[latin1]{inputenc}
\usepackage[T1]{fontenc}
\usepackage{textcomp} 
\usepackage{amsmath}
\usepackage{amsfonts}
\usepackage{amsthm}
\usepackage{mathrsfs}
\usepackage{amssymb}
\usepackage{csquotes}
\usepackage{graphicx}
\usepackage{enumerate}
\usepackage[mathscr]{euscript}
\DeclareSymbolFont{rsfs}{U}{rsfs}{m}{n}
\DeclareSymbolFontAlphabet{\mathscrsfs}{rsfs}

\theoremstyle{definition}
\newtheorem{Def}{Definition}[section]

\newtheorem{Rmk}[Def]{Remark}

\theoremstyle{plain}
\newtheorem{Prop}[Def]{Proposition}
\newtheorem{Thm}[Def]{Theorem}
\newtheorem{Lemma}[Def]{Lemma}
\newtheorem{Cor}[Def]{Corollary}

\newcommand{\R}{\mathbb{R}}

\newcommand{\Z}{\mathbb{Z}}
\newcommand{\N}{\mathbb{N}}

\newcommand{\PP}{\mathbb{P}}

\renewcommand{\epsilon}{\varepsilon}

\title{ Slow, ordinary and rapid points for Gaussian Wavelets Series and application to Fractional Brownian Motions }
\author{C. Esser\footnote{ Universit\'e de Li\`ege, D\'epartement de math\'ematique -- zone Polytech 1, 12 all\'ee de la D\'ecouverte, B\^at. B37, B-4000 Li\`ege. l.loosveldt@uliege.be} and L. Loosveldt\footnote{\underline{Corresponding author} at:  Université du Luxembourg, UR en Mathématiques, Maison du nombre, 6 avenue de la Fonte, L-4364, Esch-sur-Alzette, Luxembourg. laurent.loosveldt@uni.lu.}}

\begin{document}

\maketitle

\begin{abstract}
We study the Hölderian regularity of Gaussian wavelets series and show that they display, almost surely, three types of points: slow, ordinary and rapid. In particular, this fact holds for the Fractional Brownian Motion. We also show that this property is satisfied for a multifractal extension of Gaussian wavelet series. Finally, we remark that the existence of slow points is specific to these functions.
\end{abstract}


\section*{Introduction}

Let $B$ denote the standard Brownian motion on $\R$. The Khinchin law
of the iterated logarithm \cite{khint24} allows to control the behavior of $B$ at a
given point, in the sense that for every $t \in \R$, it holds 
\begin{equation}\label{eq:lawitlog}
\limsup_{r \to 0} \frac{|B(t+r)-B(t)|}{\sqrt{|r| \log \log |r|^{-1}}} = \sqrt{2}
\end{equation}
on an event of probability one. As a  direct application of Fubini's
theorem, one obtains that almost surely, the set of points $t \in \R$
such that \eqref{eq:lawitlog} holds, 
called ordinary points, has full Lebesgue measure. This contrasts with the uniform
H\"older condition obtained by Paul Lévy in 1937 which states that the uniform modulus of continuity of $B$ is of larger order:
almost surely, one has
$$
\limsup_{r \to 0} \sup_{ t \in [0,1]}\frac{|B(t+r)-B(t)|}{\sqrt{|r|
    \log |r|^{-1}}} = \sqrt{2}.
$$
In particular, there exist exceptional points, called fast points,  where the law of the iterated
logarithm fails. In 1974, Oray and Taylor studied how often this
exceptional behavior holds and proved
especially that the Hausdorff dimension of the set
$$
\left\{ t \in [0,1] : \limsup_{r \to 0} \frac{|B(t+r)-B(t)|}{\sqrt{|r|
    \log |r|^{-1}}} \geq \lambda \sqrt{2} \right\}
$$ 
is given almost surely by $1-\lambda^{2}$ for every $\lambda \in
[0,1]$, see \cite{oreytaylor74}.

In the meanwhile, Kahane proposed in \cite{kahane85} an easy way to study the
regularity and irregularity properties of the Brownian
motion. Its method relies on the expansion of $B$ on $[0,1]$ in the so-called Faber-Schauder system. If $\Lambda$ is the triangular function
\[ \Lambda \, : \, x \mapsto  \left\{
\begin{array}{cl}
x &\text{ if $\frac{1}{2} \leq x < 1$} \\[2ex] 
\  1-x &\text{ if $0 \leq x < \frac{1}{2}$} \\[2ex] 
\ 0 &\text{ otherwise,}
\end{array} 
\right.   \]
and $\xi$, $\xi_{j,k}$ ($j \in \N_0$ and $k \in \{0,\ldots,2^j-1\}$) are independent
$\mathcal{N}(0,1)$ random variables then, we have
\begin{align}\label{eq:intro:bfs}
B(t)=\sum_{j=0}^{+ \infty} \sum_{k=0}^{2^{j}-1} \xi_{j,k} 2^{-j/2} \Lambda(2^jt-k)+\xi t
\end{align}
where the convergence is  almost surely uniform for $t \in [0,1]$. 
Working with this expression, Kahane recovered the law of the iterated logarithm and the estimation of the
modulus of continuity of the Brownian
motion. Furthermore, Kahane obtained the existence of a third category
of points, presenting a slower oscillation. These points, called slow
points, satisfy the condition
$$
\limsup_{r \to 0} \frac{|B(t+r)-B(t)|}{\sqrt{|r|}} < + \infty. 
$$

The law of iterated logarithm and the study of the set of fast points has naturally been studied and extended
since then for more general classes of gaussian processus such as gaussian processes with stationary increments,
see e.g. \cite{marcus68,orey72,bingham86,marcusrosen92,monradrootzen95,khoshnevisanshi2000}. In
particular, given a fractional Brownian motion $B_{h}$
 of index $h \in (0,1)$, one has almost surely 
$$
\limsup_{r \to 0} \frac{|B_{h}(t+r)-B_{h}(t)|}{r^{h}\sqrt{ \log \log |r|^{-1}}} = \sqrt{2}
$$
and
$$
\limsup_{r \to 0} \sup_{ t \in [0,1]}\frac{|B_{h}(t+r)-B_{h}(t)|}{r^{h}\sqrt{
    \log |r|^{-1}}} = \sqrt{2}.
$$

In 1999, Meyer, Sellan and Taqqu introduced a famous decomposition of
the fractional Brownian motion using a Lemarié-Meyer or sufficiently smooth Daubechies
wavelet $\psi$, which decorrelates
the high frequencies \cite{meyersellantaqqu99}. More precisely, any fractional Brownian motion $B_{h}$
of Hurst index $h \in (0,1)$ can be written as
\begin{equation}\label{eq:mbf}
B_{h}(t) = \sum_{j \in \N}\sum_{k \in \Z} 2^{-hj} \xi_{j,k} \psi_{h+1/2} (2^{j}t-k)
 + R(t)
\end{equation}
where $R$ is a smooth process, $(\xi_{j,k})_{(j,k)\in\N  \times \Z}$ is a
sequence of independent $\mathcal{N}(0,1)$ random variables, and
$\psi_{\alpha}$ is defined by
$$
\hat{\psi}_{\alpha}(\xi) = \frac{1}{|\xi|^{\alpha}}\hat{\psi}(\xi).
$$
Note that such a function leads to a biorthogonal wavelet basis, see
Section \ref{sec:prelim}. 

Motivated by the study of fractional Brownian motions
using this particular decomposition, we develop in this paper a
systematic study of the different pointwise behaviors of random wavelets series of the form
\begin{equation}\label{eq:RWS}
f_h = \sum_{j \in \N} \sum_{k \in \Z}
  \xi_{j,k} 2^{-h j} \psi(2^j \cdot -k ) 
\end{equation}
where $(\xi_{j,k})_{(j,k) \in\N \times \Z}$ denote a sequence of independent
$\mathcal{N}(0,1)$ random variables, $h \in (0,1)$ is fixed and  $\psi$ is any
compactly supported or smooth wavelet, see Section \ref{sec:prelim} for a
precise definition. Note that, even if the expression \eqref{eq:RWS}
is very similar to \eqref{eq:intro:bfs}, dealing with it requires much
more technical arguments. Indeed, due to the symmetry of the function
$\Lambda$, most of the terms vanish in the expansion  of the
increments $B(t)-B(s)$ if $s$ and $t$ are closed enough. This fact can
not be used anymore while working with an arbitrary wavelet and
compensations of different terms may occurs. 

Concerning the regularity of the function $f_h$, one can show, see Proposition \ref{prop:variance} below, that, for all $s,t \in \R$,
\[ \mathbb{E}[(f_h(t)-f_h(s))^2] \leq C |s-t|^{2h}, \]
for some deterministic constant $C>0$. Applying Kolmogorov continuity
Theorem for gaussian processes, one can deduce that almost surely, for
every $t \in \R$, one has
$$
\limsup_{r \to 0} \frac{|f_{h}(t+r)-f_{h}(t)|}{r^{h-\varepsilon}} < + \infty
$$
for every $\varepsilon >0$. 
The aim of this paper is to  characterize more precisely the pointwise behavior of such a wavelet series, in the spirit of what is known for the Brownian motion. Recently, generalized H\"older spaces has been introduced to address this kind of questions \cite{kreit18,loonic21} as well as the regularity of solutions of partial differential equations \cite{loonic20}. This article is the continuation of the work done in
\cite{ayache,ayacheesserkleyntssens,kahane85} for the Brownian motion.

As a consequence of our results, we get that almost surely, the Hölder
exponent of the random wavelet series $f_h$ is $h$ while it does not
belong to the uniform Hölder space of order $h$. Nevertheless, if $t$
is a slow point, $f_h$ belongs to the \textbf{pointwise} Hölder space
of order $h$ at $t$. One can therefore wonder if this feature is
common or if it is  specific to the functions under study in this paper. To address this question, in Section \ref{sec:genericity} we recall two commonly used notions of genericity: the prevalence and the Baire category point of view. We obtain that, in both sense, the existence of slow points is a specific property of gaussian random wavelets series.

The paper is structured as follows: in Section \ref{sec:prelim}, before stating our main result, we recall the most important tools used in the paper: discrete wavelet transform and modulus of continuity as well as some fundamental inequalities. Section \ref{sec:regularity} is devoted to exploring the regularity of the gaussian random wavelet serie \eqref{eq:RWS} by identifying three precise pointwise estimates. In Section \ref{sec:wavcoef}, we focus on the irregularity that we deduced from the asymptotic behavior of the wavelet coefficients. Then, using the obtain results of regularity and irregularity, we prove our main result concerning the existence of slow, ordinary and rapid points in Section \ref{sec:proof}. This fact is extended to multifractal gaussian random series in Section \ref{sec:multi}. Finally, Section \ref{sec:genericity}, focuses on the results of generecity for slow points. In this paper, $C$ stands for a deterministic constant not necessary the same in different lines.

\section{Preliminaries and statement of the main result}\label{sec:prelim}

In this section, we present the notions needed for the statement of the
main theorem and the first result concerning the convergence of the
wavelets series defined in~\eqref{eq:RWS}.

Let us first briefly recall some definitions and notations about wavelets
and biortogonal wavelets (for more precisions, see e.g.\
\cite{Daubechies:92,Meyer:95,Mallat:99,Cohen:92}). Under some general
assumptions, there exist two functions $\phi$ and $\psi$, called
wavelets, which generate two orthonormal bases of $L^{2}(\R)$, namely
\[
 \{\phi(\cdot-k)\}_{k\in\Z}\cup\{\psi(2^j \cdot -k ) : j \in \N,  k\in \Z\}
\]
and
\[
 \{\psi(2^j \cdot-k): j\in\Z, k\in \Z \}.
\]
Any function $f\in L^2(\R)$ can be decomposed as follows,
\[
f=\sum_{k\in \Z} C_k \phi(\cdot -k) + \sum_{j\in \N} \sum_{k\in\Z}
c_{j,k} \psi(2^j \cdot -k) = \sum_{j \in \Z}\sum_{k \in \Z} c_{j,k} \psi(2^j
\cdot-k)
\]
where
\[
c_{j,k}=2^{j}\int_{\R}f(x) \psi(2^jx-k)\, dx
\]
and
\[
 C_k=\int_{\R^n} f(x) \phi(x-k)\, dx.
\]
Let us remark that we do not choose the $L^2(\R)$ normalization for
the wavelets, but rather an $L^\infty(\R)$ normalization, which is better
fitted to the study of the H\"olderian regularity. Amoung the families
of wavelet basis that exist, we are mostly interested in two classes:
The Lemarié-Meyer wavlets for which  
$\phi$ and $\psi$ belong to the Schwartz class
$\mathcal{S}(\R)$, or Daubechies wavelets for which  $\phi$ and $\psi$ are
compactly supported functions. In both cases, the first moment of the
wavelet $\psi$ vanishes.

The setting in which we work is more general that just
orthogonal wavelet basis, so that it allows to cover the important
example supply by the fractional Brownian motion. Biorthogonal wavelet
bases are a couple of two Riesz wavelet bases of $L^{2}(\R)$ generated respectively by $\psi$ and $\widetilde{\psi}$ and such that
\[
2^{j/2} 2^{j'/2} \int_{\R} \psi(2^jx-k) \widetilde{\psi}(2^{j'}x-k') dx = \delta_{j,j'} \delta_{k,k'}.
\]
In that case, any function $f \in L^2(\R)$ can be decomposed as
\[
f = \sum_{j \in \Z} \sum_{k \in \Z} c_{j,k} \psi(2^j \cdot -k ) 
\]
where
\[
c_{j,k} = 2^j \int_{\R} f(x) \widetilde{\psi}(2^{j}x-k) dx.
\]
Wavelet basis and biorthogonal wavelet basis give a powerful tool to study the regularity and irregularity of functions or signals belonging to numerous functional spaces, see e.g. \cite{jaffardmeyer96,Meyer:95,Bastin:14,Jaffard:04,aubrybastin:10,Jaffard:04b,Clausel:11,esserkleynic,kreit18,loonic21}. In this paper, they will be used to get irregularity properties in Section \ref{sec:wavcoef}.


\medskip

Let us now present two lemmata that allow to prove the uniform
convergence on any compact set  of the series defined in \eqref{eq:RWS}, where $\psi$
comes from any wavelet basis or biorthogonal wavelet basis. The first one is very
classical and the second one gives informations about the asymptotic
behavior of a sequence of i.i.d gaussian random variables. 

\begin{Lemma}\label{unibound}
There exists a constant $C_1>0$ such that, for all $x \in \R^d$
\[\sum_{k \in \Z} \frac{1}{(1+|x-k|)^4}  \leq C_1.\]
\end{Lemma}

\begin{Lemma}\label{boundnormal} \cite{ayachebertrand,ayachetaqqu}
Let $(\xi_{j,k})_{(j,k) \in \N \times \Z}$ be a sequence of independent $\mathcal{N}(0,1)$ random variables. There are an event $\Omega^*$ of probability $1$ and a positive random variable $C_2$ of finite moment of every order such that, for all $\omega \in \Omega^*$ and $(j,k) \in \Z^2$, the inequality
\begin{align}\label{eq:boundnorm}
|\xi_{j,k}(\omega)| \leq C_2(\omega) \sqrt{\log(3+j+|k|)}
\end{align}
holds.
\end{Lemma}

Let $f_h$ denote the process defined in \eqref{eq:RWS}. 
For all $j \in \N$, we set
\[ f_{h,j} = \sum_{k \in \Z} \xi_{j,k} 2^{-h j} \psi(2^j \cdot -k ).\]
and remark that, if inequality \eqref{eq:boundnorm} holds then, thanks
to Lemma \ref{unibound} and the fast decay of the wavelet and its
derivative (or  using the compactness of the support of $\psi$), the sum in the right-hand side converges uniformly on any compact set, as well as the sum
\begin{align}\label{derioffj}
\sum_{k \in \Z} \xi_{j,k} 2^{(1-h) j} D\psi(2^j \cdot -k ).
\end{align}
Therefore, for all $j$, $f_{h,j}$ is continuously differentiable with
derivative $Df_{h,j}$ given by~\eqref{derioffj}. In particular, the
process $f_{h}$ is well defined and bounded on the event $\Omega^*$ of
probability~$1$. In the following, we will hence work on this event
without mentioning it explicitly. 

Using similar arguments, the following Proposition gives us a first information concerning the regularity of the function $f_h$.

\begin{Prop}\label{prop:variance}
Let $f_h$ be the random wavelet series defined in \eqref{eq:RWS}. If $\psi$ is continously differentiable, there exists a constant $C>0$ such that, for all $s,t \in \R$
\[ \mathbb{E}[(f_h(t)-f_h(s))^2] \leq C |s-t|^{2h}. \]
\end{Prop}

\begin{proof}
From the independence of the centered random variables in $(\xi_{j,k})_{(j,k) \in \N \times \Z}$, we have
\[\mathbb{E}[(f_h(t)-f_h(s))^2] = \sum_{j \in \N} \sum_{k \in \Z} 2^{-2hj} (\psi(2^j t -k)-\psi(2^j s-k))^2. \]
Let us fix $t$ and assume that  $s$ and $\nu \in \Z$ are such that
\[ 2^{-\nu} < |t-s| \leq 2^{-\nu +1}.\] 
For all $j \leq \nu$, we set
\[ F_{j,t} \, : \, x \mapsto \sum_{k \in \Z} 2^{-2hj} (\psi(2^j t -k)-\psi(2^j x-k))^2.\]
Let us then remark that $F_{j,t}$ is continuously differentiable on any compact set. If $\psi$ is compactly supported, this is obvious. Otherwise, it comes from the fast decay of $\psi$ and Lemma \ref{unibound}. In any case, by the mean value theorem, there exist $x_1$ between $s$ and $t$ such that
\begin{align*}
|F_{j,t}(s)-F_{j,t}(t) |&= \left| (s-t)\sum_{k \in \Z} 2^{(1-2h)j} 2 (\psi(2^j t -k)-\psi(2^j x_1-k)) D \psi(2^jx_1-k) \right| \\
& \leq C \left| (s-t)\sum_{k \in \Z} 2^{(1-2h)j} (\psi(2^j t -k)-\psi(2^j x_1-k)) \right|.
\end{align*}
Applying again the mean-value theorem to the function
\[ g_j \, : \, x \mapsto \sum_{k \in \Z} 2^{(1-2h)j} \psi(2^j x -k),\]
we get
\[ |F_{j,t}(s)-F_{j,t}(t) | \leq C |s-t|^2  2^{(2-2h)j}\]
because
\[\sum_{k \in \Z} D\psi(2^j t -k) \]
can be uniformly bounded, by the fast decay and Lemma \ref{unibound} or using the compact support of $\psi$. With the same argument, we also get, for all $j>\nu$,
\begin{align*}
\left|  \sum_{k \in \Z} 2^{-2hj} (\psi(2^j t -k)-\psi(2^j s -k))^2
  \right|  & \leq   2^{-2hj} \sum_{k \in \Z}4 ((\psi(2^j t -k))^2+(\psi(2^j s -k))^2) \\
& \leq C 2^{-2hj}.
\end{align*}
Putting all of these together, we have
\begin{align*}
\mathbb{E}[(f_h(t)-f_h(s))^2] &= \sum_{j \leq \nu}  \sum_{k \in \Z} 2^{-2hj} (\psi(2^j t -k)-\psi(2^j s-k))^2 \\
&+\sum_{j > \nu}  \sum_{k \in \Z} 2^{-2hj} (\psi(2^j t -k)-\psi(2^j s-k))^2 \\
& \leq C (|s-t|^2 \sum_{j \leq \nu} 2^{(2-2h)j} +  \sum_{j > \nu}2^{-2hj} )\\
& \leq C( |s-t|^2 2^{(2-2h)\nu}+2^{-2h \nu}) \\
& \leq C |s-t|^{2h}.
\end{align*}
\end{proof}

From the last proposition and Kolmogorov continuity Theorem for
gaussian processes, we know that the sample path of $f_h$ are almost surely locally
Hölder-continuous of order $h-\varepsilon$ for every $\varepsilon >0$. Our aim in this paper is to give more precise information concerning the regularity of $f_h$. In order to state it, we recall finally that a modulus of continuity is an increasing function $\omega:\R^+ \to  \R^+$ satisfying 
 $\omega (0) = 0$  and for which there is $C>0$ such that 
 $\omega (2x) \leq C \omega (x)  $
for all $x \in \R^{+}$. 
Wavelet characterizations of regularity require the  following
additional regularity property for moduli of continuity, see
\cite{jaffardmeyer96}: A modulus of continuity $\omega$ is regular if
there is $N \geq 0 $ such that 
\begin{equation} \label{modreg}
 \begin{cases} & \displaystyle \sum_{j=J}^\infty 2^{Nj}
            \omega (2^{-j})  \leq C 2^{NJ} \omega (2^{-J}) \\ 
           &\!\! \displaystyle \sum_{j=-\infty}^J 2^{(N+1)j} \omega
            (2^{-j})  \leq C 2^{(N+1)J} \omega
                       (2^{-J})    \end{cases} 
\end{equation}
for all $J \geq 0$. 
Our main result will use three different regular moduli of continuity: 
\begin{itemize}
\item the modulus of continuity $\omega_{r}$ of the rapid points is
  defined by
$$
\omega_{r}^{(h)}(x) = |x|^{h}\sqrt{\log|x|^{-1}} 
$$ 
\item the modulus of continuity $\omega_{o}$ of the ordinary points  is
  defined by
$$
\omega_{o}^{(h)}(x) = |x|^{h}\sqrt{\log\log|x|^{-1}}
$$ 
\item the modulus of continuity $\omega_{s}$  of the slow points is
  defined by
$$
\omega_{s}^{(h)}(x) = |x|^{h}.
$$ 
\end{itemize}

\begin{Thm}\label{thm:main}
Almost surely, the random wavelets series  defined in \eqref{eq:RWS} satisfies the following property for every non-empty interval $I$ of $\R$:
\begin{itemize}
\item  For almost every $t\in I$, 
\begin{equation}\label{thm:main1a}
 \limsup_{s \to t} \dfrac{|f_h(s)-f_h(t) |}{\omega_{o}^{(h)}(|s-t|)}< + \infty 
\end{equation}
and if $\omega$ is a modulus of continuity such that $\omega =
o(\omega_{o}^{(h)})$, then
\begin{equation}\label{thm:main1b}
 \limsup_{s \to t} \dfrac{|f_h(s)-f_h(t) |}{\omega(|s-t|)}= + \infty ,
\end{equation}
Such points  are called \emph{ordinary points}.
\item There exists $t \in I$  such that
\begin{equation}\label{thm:main2a}
 \limsup_{s \to t} \dfrac{|f_h(s)-f_h(t) |}{\omega_{r}^{(h)}(|s-t|)}< + \infty 
\end{equation}
and if $\omega$ is a modulus of continuity such that $\omega =
o(\omega_{r}^{(h)})$, then
\begin{equation}\label{thm:main2b}
 \limsup_{s \to t} \dfrac{|f_h(s)-f_h(t) |}{\omega(|s-t|)}= + \infty ,
\end{equation}
Such points  are called \emph{rapid points}.

\item There exists $t \in I$  such that
\begin{equation}\label{thm:main3a}
\limsup_{s \to t} \dfrac{|f_h(s)-f_h(t) |}{\omega_{s}^{(h)}(|s-t|)}< + \infty .
\end{equation}
and if $\omega$ is a modulus of continuity such that $\omega =
o(\omega_{s}^{(h)})$, then
\begin{equation}\label{thm:main3b}
 \limsup_{s \to t} \dfrac{|f_h(s)-f_h(t) |}{\omega(|s-t|)}= + \infty .
\end{equation}
Such points  are called \emph{slow points}.
\end{itemize}
\end{Thm}

\begin{Rmk}
Theorem \ref{thm:main} is stated in full generality but let us already emphasize that it can be improved while considering compactly supported wavelets. Indeed, in this case, one can show the strict positiveness of the limits in \eqref{thm:main1a}, \eqref{thm:main2a} and \eqref{thm:main3a}, see Remark \ref{rem:suppcompact} below.
\end{Rmk}

Theorem \ref{thm:main} generalises the  famous result of Kahane about the different pointwise behaviors of the Brownian motion, see \cite[Theorem 3 in Chapter
16]{kahane85}. Moreover, it can be applied to the random wavelet
series in~\eqref{eq:mbf}, hence it proves the existence of the three
types of points for the fractional Brownian motion. 

\begin{Rmk}
A good tool to determine the regularity of a locally bounded function $f$ at a point $t$ is to compute its Hölder exponent $h_f(t)$. If $\alpha >0$, we say that $f$ belongs to the \textit{pointwise Hölder space} $C^\alpha(t)$ if there exists a polynomial $P_t$ of degree strictly less than $\alpha$ and a constant $C>0$ such that, for all $j \in \N$,
\[ \sup_{x \in B(x_0,2^{-j})} | f(x)-P_t(x) | \leq C 2^{-\alpha j}. \]
It is straightforward to show that, as soon as $\alpha < \beta$, $C^\beta (t) \subseteq C^\alpha(t)$. Therefore, we define the \textit{pointwise Hölder exponent} of $f$ at point $t$ by
\[ h_f(t) = \sup \{ \alpha >0 \, : \, f \in C^\alpha(t) \}.\]
Theorem \ref{thm:main} tells us in particular that, if $t$ is an ordinary or rapid point, then, for all $\varepsilon>0$, $f_h \in C^{h-\varepsilon}(t)$ and $f_h \notin C^{h}(t)$ which gives $h_{f_h}(t)=h$. On the contrary, if $t$ is a slow point, $f_h \in C^h(t)$ while, for all $\varepsilon>0$, $f_h \notin C^{h+\varepsilon}(t)$ and so, again, $h_{f_h}(t)=h$.
In fact, the finiteness of the limits \eqref{thm:main1a}, \eqref{thm:main2a} and \eqref{thm:main3a} means that the function $f$ belongs to a generalized pointwise H\"older space \cite{kreit18,loonic21} associated to the corresponding modulus of continuity. 
\end{Rmk}

The next sections are dedicated to the proof of this result. Any open
interval in $\R$ can be written as a countable union of dyadic
intervals. Then, to prove Theorem \ref{thm:main}, it is sufficient to
show that, for all dyadic interval of the form
$\lambda_{j,k}=[k2^{-j}, (k+1)2^{-j}[$ with $j \in \N, k \in \Z$,
there exist an event $\Omega_{j,k}$ of probability $1$ such that, for
all $\omega \in \Omega_{j,k}$, almost every $t \in \lambda_{j,k}$ is
ordinary and there exist $t_r \in \lambda_{j,k}$ which is rapid and
$t_s \in \lambda_{j,k}$ which is slow. For the sake of simpleness in
notation, we will only do the proofs in full details for
$\lambda_{0,0}=[0,1)$. In fact, after dilatations and translations,
our proofs hold true for any arbitrary dyadic interval. Note also that
the proofs will be
done in the case where the wavelet $\psi$ is in the Schwartz class;
it can easily be adapted and simplified if $\psi $ is compactly
supported.

\section{Regularity properties} \label{sec:regularity}

In this section, we establish inequalities \eqref{thm:main1a}, \eqref{thm:main2a} and \eqref{thm:main3a}. Concerning rapid and ordinary points, the conduct of the proof is similar. First we use Lemma \ref{boundnormal} to bound the coefficients in \eqref{eq:RWS}. Then we use the fast decay of the wavelet to measure the contribution of the coefficients associated to dyadic intervals that are far away from the point of interest $t$ in the difference $|f_{h}(s)-f_{h}(t)|$. Concerning the slow points, we take advantage of a procedure initiated by Kahane in \cite{kahane85} to identify points for which we can obtain more precise information concerning the coefficients of the ``closest'' intervals while we still use Lemma \ref{boundnormal} and the fast decay for the ``furthest'' one.

\subsection{Rapid points}

To prove the existence of rapid points, we apply Lemma \ref{boundnormal} and get an uniform modulus of continuity for the function $f_{h}$. We deal with the coefficients associated to furthest intervals thanks to the following lemma.

\begin{Lemma}\label{lemma:rapid}
There exists a constant $C_1>0$ such that, for all $j \in \N_0$ and $x \in (0,1)$,
\[ \sum_{|k| > 2^{j+1}} \frac{\sqrt{\log(3+j+|k|)}}{(1+|2^jx-k|)^5} \leq C_1\]
\end{Lemma}

\begin{proof}
Let us fix $j \in \N_0$ and $x \in (0,1)$. As $|k| > 2^{j+1}$, obviously, $\frac{k}{2^j} \notin (-2,2)$ so let $n \notin \{0, \pm 1, - 2\}$ be such that $n2^j \leq k < (n+1)2^j$. Now, as $x \in (0,1)$, we have
\[ |2^jx-k| \geq 2^j |x-n| -1 \geq 2^j (|n|-1)-1 \geq \frac{|k|}{2}-1.\]
Thus, for all such $j,k$ and $x$, we have
\[ \frac{\sqrt{\log(3+j+|k|)}}{(1+|2^jx-k|)} \leq 2 \frac{\sqrt{\log(3+2|k|)}}{|k|}\]
and we conclude using the boundedness of the function $x \mapsto \frac{\sqrt{\log(3+2 x)}}{x}$ on $[1,+ \infty[$ and Lemma \ref{unibound}.
\end{proof}

\begin{Prop} \label{prop:rapid}
Almost surely, there exists a constant $C_1>0$ such that, for all $t,s \in (0,1)$ we have
\[ |f_h(s)-f_h(t)| \leq C_1 |t-s|^h \sqrt{\log|t-s|^{-1}}.\]
In particular,  one has almost surely
\[
 \limsup_{s \to t} \dfrac{|f_h(s)-f_h(t) |}{\omega_{r}^{(h)}(|s-t|)}< + \infty 
\]
for every $t\in (0,1)$.
\end{Prop}

\begin{proof}
Let us assume that  $t,s$ and $\nu \in \N$ are such that
\[ 2^{-\nu} < |t-s| \leq 2^{-\nu +1}.\]
For all $j \leq \nu$, by the mean value theorem, there exists $x$ between $s$ and $t$ such that
\begin{align}\label{eq:demorapid1}
|f_{h,j}(t)-f_{h,j}(s)| \leq |t-s| |D f_{h,j} (x)| 
\end{align}
and, using the fast decay of $D\psi$ and \eqref{eq:boundnorm}, it
follows that
\begin{align*}
|D f_{h,j}(x)| &\leq C C_2 2^{(1-h)j} \left( \sum_{|k| \leq 2^{j+1}} \frac{\sqrt{\log(3+j+|k|)}}{(1+|2^jx-k|)^4} \right. \\ & \left. \quad \quad + \sum_{|k| > 2^{j+1}} \frac{\sqrt{\log(3+j+|k|)}}{(1+|2^jx-k|)^5} \right).
\end{align*}
We bound the second sum by Lemma \ref{lemma:rapid} while, for the first sum, we have
\[ \sum_{|k| \leq 2^{j+1}} \frac{\sqrt{\log(3+j+|k|)}}{(1+|2^jx-k|)^4} \leq C  \sum_{|k| \leq 2^{j+1}} \frac{\sqrt{j}}{(1+|2^jx-k|)^4} \leq C \sqrt{j}\]
by Lemma \ref{unibound}. Thus, we obtain
\begin{align*}
| \sum_{j \leq \nu} (f_{h,j}(t)-f_{h,j}(s))| &\leq C C_2|t-s|2^{(1-h)\nu } \sqrt{\nu} \\ &\leq C C_2(\omega)|t-s|^h \sqrt{\log|t-s|^{-1}}. 
\end{align*}
Now, if $j > \nu$, by splitting the sums in the same way, we also have
\[ |f_{h,j}(t)| \leq C C_2 2^{-h j} \sqrt{j} \text{ and } |f_{h,j}(s)| \leq C C_2 2^{-h j} \sqrt{j}\]
which obviously leads to
\[ |\sum_{j > \nu} f_{h,j}(t)| \leq C C_2|t-s|^h \sqrt{\log|t-s|^{-1}} \]
and
\[ |\sum_{j > \nu} f_{h,j}(s)| \leq C C_2|t-s|^h \sqrt{\log|t-s|^{-1}}.\]
The conclusion follows immediately.
\end{proof}



\subsection{Ordinary points}

To establish the existence of ordinary points, we need to introduce
some notations. If $j \in \N_0$, and $t \in (0,1)$, we denote by
$k_j(t)$ the unique positive integer in $\{0,\ldots,2^j-1\}$ such that
$t \in [k_j(t)2^{-j},(k_j(t)+1)2^{-j})$. We also define, for all $n \in \N$ the set
\[ \kappa_j^t(n) = \{ k \in \Z \, : \, |k-k_j(t)| \leq n \} \]
The main proof of this section consists in applying Lemma \ref{boundnormal} to a sequence of independent $\mathcal{N}(0,1)$ random variables indexed from $k_j(t)$. Then, Lemmata \ref{lemma:ordinary} and \ref{lemma:ordinary2} help us to deal with the coefficients associated with the furthest intervals.

\begin{Lemma}\label{lemma:ordinary}
There exists a constant $C_1>0$ such that, for all $j \in \N_0$, $t
\in (0,1)$ and $x \in (0,1)$ such that $|t-x| \leq 2^{-j+1}$ we have
\[ \sum_{k \notin \kappa_j^t(j)} \frac{\sqrt{\log(3+j+|k-k_j(t)|)}}{(3+|2^jx-k|)(1+|2^jx-k|)^4} \leq C_1\]
\end{Lemma}

\begin{proof}
Let us fix $j \in \N_0$, $t \in (0,1)$ and $x \in (0,1)$ such that $|t-x| \leq 2^{-j+1}$. If  \linebreak $k \notin \kappa_j^t(j)$, we immediately have $|2^jx-k_j(t)| \leq 3$ and thus $|2^j x-k| \geq |k_j(t)-k|-3$. Therefore,
\[ \frac{\sqrt{\log(3+j+|k-k_j(t)|)}}{(3+|2^jx-k|)} \leq \frac{\sqrt{\log(3+2|k-k_j(t)|)}}{|k_j(t)-k|}\]
and the conclusion follows just like in Lemma \ref{lemma:rapid}.
\end{proof}

\begin{Lemma}\label{lemma:ordinary2}
There exists a constant $C_1>0$ such that, for all $j \in \N_0$, $t
\in (0,1)$ and $s \in (0,1)$ such that $|2^js-k_j(t)| \leq j$ we have
\[ \sum_{k \notin \kappa_j^t(2j)} \frac{\sqrt{\log(3+j+|k-k_j(t)|)}}{(1+|2^js-k|)^5} \leq C_1\]
\end{Lemma}

\begin{proof}
Let us fix $j \in \N_0$ $t \in (0,1)$ and $s \in (0,1)$ such that $|2^js-k_j(t)| \leq j$. If $k \notin \kappa_j^t(2j)$, of course, $|2^js-k_j(t)| <  \frac{|k-k_j(t)|}{2}$ and thus $|2^j s-k| \geq \frac{|k-k_j(t)|}{2}$. It follows that
\[ \frac{\sqrt{\log(3+j+|k-k_j(t)|)}}{1+|2^js-k|} \leq 2 \frac{\sqrt{\log(3 + 2|k-k_j(t)|)}}{|k-k_j(t)|}\]
and, again, we conclude just like in Lemma \ref{lemma:rapid}.
\end{proof}

\begin{Prop}\label{prop:ordinary}
Almost surely, for almost every $t \in (0,1)$, 
\[  \limsup_{s \to t} \dfrac{|f_h(s)-f_h(t) |}{\omega_{o}^{(h)}(|t-s|)}< + \infty. \]
\end{Prop}

\begin{proof}
Let us fix $t \in (0,1)$, and $j \in \N_0$ and let $0 \leq k_j(t) <
2^{j}$ be such that $t \in [k_j(t)2^{-j},(k_j(t)+1)2^{-j}[$. A simple
modification of Lemma \ref{boundnormal} insure us the existence of a
positive random variable $C_t$ of finite moment of every order such
that, almost surely
\[ |\xi_{j,k}| \leq C_t \sqrt{\log(3+ j+|k-k_j(t)|)}.\]

As previously, if $s \in (0,1)$ and $\nu \in \N$ is such that
\[ 2^{-\nu} < |t-s| \leq 2^{-\nu +1},\]
we first start by considering, for all $j \leq \nu$ ,$|D f_{h,j} (x)|$ for a $x$ between $s$ and $t$. In this case, we have
\begin{align*}
|Df_{h,j} (x)| &\leq C C_t 2^{(1-h)j} \left( \sum_{k \in \kappa_j^t(j)} \frac{\sqrt{\log(3+j+|k-k_j(t)|)}}{(1+|2^jx-k|)^4} \right. \\& \left.  \quad \quad + \sum_{k \notin \kappa_j^t(j)} \frac{\sqrt{\log(3+j+|k-k_j(t)|)}}{(3+|2^jx-k|)(1+|2^jx-k|)^4} \right) 
\end{align*}
and we bound the second sum by Lemma \ref{lemma:ordinary}, while, for the first sum, we have
\[ \sum_{k \in \kappa_j^t(j)} \frac{\sqrt{\log(3+j+|k-k_j(t)|)}}{(1+|2^jx-k|)^4} \leq C \sum_{k \in \kappa_j^t(j)} \frac{\sqrt{\log(j)}}{(1+|2^jx-k|)^4} \leq C \sqrt{\log(j)} \]
by Lemma \ref{unibound}.
These inequalities lead to 
\begin{equation}\label{eq:ordinary1}
| \sum_{j \leq \nu} (f_{h,j} (t)-f_{h,j} (s)) | \leq C C_t |t-s|^h
\sqrt{\log \log|t-s|^{-1}}. 
\end{equation}

To bound $|f_{h,j} (t)|$, for all $j > \nu$, we use the same techniques and get
\[ |\sum_{j > \nu} f_{h,j} (t)| \leq C C_t |t-s|^h \sqrt{\log \log|t-s|^{-1}}. \]

The bound for $|f_{h,j} (s)|$ is a little bit more tricky. As $|2^js-k_j(t)| \leq 2^{j-\nu+2}$, we first consider the case when $2^{j- \nu+2} \leq j$, then
\begin{align*}
|f_{h,j} (s)| & \leq C C_t 2^{-h j} \left( \sum_{k \in \kappa_j^t(2j)} \frac{\sqrt{\log(3+j+|k-k_j(t)|)}}{(1+|2^js-k|)^4} \right. \\ & \left. \quad \quad + \sum_{k \notin \kappa_j^t(2j)} \frac{\sqrt{\log(3+j+|k-k_j(t)|)}}{(1+|2^js-k|)^5} \right). 
\end{align*}
The second sum is this time bounded using Lemma \ref{lemma:ordinary2} while we again use Lemma \ref{unibound} to get
\[ \sum_{k \in \kappa_j^t(2j)} \frac{\sqrt{\log(3+j+|k-k_j(t)|)}}{(1+|2^js-k|)^4} \leq C \sqrt{\log(j)}.\]
Now, if $j < 2^{j- \nu+2}$,
\begin{align*}
|f_{h,j} (s)| & \leq C C_t 2^{-h j} \left( \sum_{k \in \kappa_j^s(2^{j- \nu+2})} \frac{\sqrt{\log(3+j+|k-k_j(t)|)}}{(1+|2^js-k|)^4} \right. \\ & \left. \quad \quad + \sum_{k \notin \kappa_j^s(2^{j- \nu+2})} \frac{\sqrt{\log(3+j+|k-k_j(t)|)}}{(1+|2^js-k|)^5} \right) 
\end{align*}
and, as $|k_j(s)-k_j(t)| \leq 2^{j-\nu+1}+2$, for all $k \in \kappa_j^s(2^{j- \nu+2})$ we have
\[ \log(3+j+|k-k_j(t)|) \leq \log(3+j+|k_j(s)-k_j(t)|+|k-k_j(s)|) \leq \log(4 2^{j-\nu+2})\]
while, for all $k \notin \kappa_j^s(2^{j- \nu+2})$,
\begin{align*}
\frac{\sqrt{\log(3+j+|k-k_j(t)|)}}{(1+|2^js-k|)^5} &\leq \frac{\log(3+j+|k_j(s)-k_j(t)|+|k-k_j(s)|)}{|k-k_j(s)|} \\
& \leq \frac{\log(4|k-k_j(s)|)}{|k-k_j(s)|} \\
& \leq C
\end{align*}
which gives us
\begin{equation}\label{eq:ordinary2}
 |f_{h,j} (s)| \leq  C C_t 2^{-h j} \sqrt{j-\nu}.\end{equation}
In total, using \eqref{eq:ordinary1} and \eqref{eq:ordinary2},  we get
\begin{align*}
|\sum_{j > \nu} f_{h,j} (s)| &\leq \sum_{j > \nu} C C_t  2^{-h j} \sqrt{j-\nu} \sqrt{\log(j)} \\
& \leq  C C_t |t-s|^h \sqrt{\log \log|t-s|^{-1}}. 
\end{align*}
The conclusion comes from Fubini's Theorem.
\end{proof}

\subsection{Slow points}

\subsubsection{An iterative procedure for slow points}

The following procedure is inspired by the one initially described by Kahane in \cite{kahane85} to identify slow points for the Brownian motion. We  generalized it here by introducing an extra parameter $m \in \N$ in order to use it for any\footnote{Kahane only considered the case $h=\frac{1}{2}$ and so its construction is made with $m=3$ which is sufficient to fulfill the condition $h> \frac{1}{m}$ we will need afterward.} $h>0$. Some clarifications are also made.

If $\mu \in \N$, $\xi \sim \mathcal{N}(0,1)$ and $l\in \N_0 $, we note 
\[p_l(\mu) = \mathbb{P}(2^l \mu < |\xi| \leq 2^{l+1} \mu).\]
For all $j,l \in \N_0$ and $0 \leq k < 2^j$, we define
\[ S_{j,l}^\mu = \{ k \in \{0,\ldots,2^{j}-1 \} \, : \, 2^l
  \mu < |\xi_{j,k}| \leq 2^{l+1} \mu \}\]
and 
\[ \Lambda_{j,l}(k) = \{ k' \in \{0,\ldots,2^{j}-1 \} \, : \, |k-k'| \leq 2^{ml} \}.\]
Note that $\#\Lambda_{j,l}(k)  \leq 2^{ml+1}+1$.
For all $j\in \N_{0}$, we define a closed set from the dyadic intervals $[k2^{-j},(k+1)2^{-j}]$ for which $k$ belongs to the set
\[ I_j^\mu=\{k \in \{0,\ldots,2^{j}-1 \} \, : \, \forall l \in \N_0 , \, \Lambda_{j,l}(k) \cap S_{j,l}^\mu = \emptyset \},\]
namely, we consider
\[ F_j^\mu = \bigcup_{k \in I_j^\mu } [k2^{-j},(k+1)2^{-j}]. \]
We want to show that, almost surely, there exists $\mu \in \N$ such that
\[ S_{\text{low}}^\mu = \bigcap_{j \in \N_{0}} F_j^\mu \neq \emptyset \]
which is equivalent to the fact that, for all $J \in \N_0$,
\[ S_{\text{low},J}^\mu= \bigcap_{j \leq J} F_j^\mu \]
is non-empty, as $(S_{\text{low},J}^\mu)_{J \in \N_{0}}$ is a decreasing sequence of compact sets. In other words, if we denote by $N_J^\mu$ the number of subintervals of $S_{\text{low},J}^\mu$, we want to show that
\[\mathbb{P}( \bigcup_{\mu \in \N} \bigcap_{J \in \N_{0}} \{N_J^\mu \geq 1\}) =1. \] 

For this purpose, we will consider sufficiently large $\mu$ such that the inequality
\begin{align} \label{condition}
\sum_{l=0}^{+ \infty} (2^{ml+1}+1) (p_l(\mu) + l \sqrt{p_l(\mu)(1-p_l(\mu))} ) < \frac{1}{4}
\end{align}
holds. Moreover, if $J_1$ is fixed, let us remark that the construction of $S_{\text{low},J_1}^\mu$ needs to consider at most $\sum_{j=0}^{J_1} 2^j$ i.i.d. $\mathcal{N}(0,1)$ random variables and, by increasing $\mu$ if necessary, one can choose to remove the intervals $[0,2^{-J_1}]$ and $[1-2^{-J_1},1]$ from $S_{\text{low},J_1}^\mu$ at this step if necessary while making $\mathbb{P}(N_{J_1}^\mu \geq (\frac{3}{2})^{J_1})$ as close to $1$ as we want. This trick helps us to start our construction with a arbitrary close to $1$ ``initial value'' of probability and to make sure that, at the end, the resulting points will differ from $0$ and $1$.

\begin{Lemma}
For all $\mu \in \N$ sufficiently large such that condition
\eqref{condition} holds, the sequence $(N_J^\mu)_{J \in \N_{0}}$ of random variables satisfies the formula
\begin{align}\label{induction}
\mathbb{P}(N_{J+1}^\mu \geq (\frac{3}{2})^{J+1}) \geq
  (1-(\frac{2}{3})^J) \mathbb{P}(N_{J}^\mu \geq (\frac{3}{2})^{J}) ,
  \quad \forall J \in \N_{0}. 
\end{align}
\end{Lemma}

\begin{proof}
For all $J$, we define\footnote{We obviously have $N_J^\mu = \# \mathcal{I}_J^\mu$.}
\[  \mathcal{I}_J^\mu = \{k \in \{0,\ldots,2^{J}-1 \} \, : \, [k2^{-J},(k+1)2^{-J}] \subseteq S_{\text{low},J}^\mu \}, \]
and remark that
$\mathcal{I}_{J+1}^\mu$ is obtained by removing from
$2\mathcal{I}_J^\mu \cup 2\mathcal{I}_{J}^\mu +1$ the elements $k \in \{0,\ldots,2^{J+1}-1\}$ such that $\Lambda_{J+1,l}(k) \cap S_{J+1,l}^\mu \neq \emptyset$ for a $l \in \N_0$. But now, if $N_J^\mu= N$, for all such a $l$,
\[ \#(S_{J+1,l}^\mu \cap (2\mathcal{I}_J^\mu \cup 2\mathcal{I}_{J}^\mu +1)) \sim \mathcal{B}(2N,p_l(\mu))  \]
so we have, by Tchebycheff's inequality, that
\[ \#(S_{J+1,l}^\mu \cap (2\mathcal{I}_J^\mu \cup 2\mathcal{I}_{J}^\mu
  +1)) \leq 2N (p_l(\mu) + (l+1) \sqrt{p_l(\mu)(1-p_l(\mu))} )\]
with probability greater than $1-(l+1)^{-2}N^{-1}$. 
Thus the number of removed $k$ is bounded by
\[ 2N \sum_{l=0}^{+ \infty} (2^{ml+1}+1) (p_l(\mu) + l \sqrt{p_l(\mu)(1-p_l(\mu))} ) \]
with probability greater than $1-N^{-1}$.
Now, with condition \eqref{condition}, we get, for all $J$,
\[ \mathbb{P}(N_{J+1}^\mu \geq \frac{3}{2}N_J | N_J^\mu =N) \geq 1-N^{-1}\]
which gives us
\begin{align*}
\mathbb{P}(N_{J+1}^\mu \geq (\frac{3}{2})^{J+1}) &\geq \mathbb{P}((N_{J+1}^\mu \geq \frac{3}{2} N_J^\mu) \cap( N_J^\mu \geq (\frac{3}{2})^J))\\
& = \sum_{N \geq (\frac{3}{2})^J} \mathbb{P}(N_{J+1}^\mu \geq \frac{3}{2}N_J | N_J^\mu =N) \mathbb{P}(N_J^\mu =N) \\
& \geq \sum_{N \geq (\frac{3}{2})^J} (1-N^{-1} )\mathbb{P}(N_J^\mu =N) \\
& \geq (1-(\frac{2}{3})^J)\sum_{N \geq (\frac{3}{2})^J} \mathbb{P}(N_J^\mu =N)\\
& = (1-(\frac{2}{3})^J) \mathbb{P}(N_{J}^\mu \geq (\frac{3}{2})^{J}).
\end{align*}
\end{proof}

\begin{Prop}\label{kahaneas}
Almost surely, there exists $\mu \in \N$ such that $(0,1) \cap S_{\text{low}}^\mu \neq \emptyset$.
\end{Prop}

\begin{proof}
From formula \eqref{induction}, if $\mu$ is large enough we deduce by induction that for all $J_1, J$ with $J_1 \leq J$,
\begin{align*}
\mathbb{P}(N_{J}^\mu \geq 1) &\geq \mathbb{P}(N_{J}^\mu \geq (\frac{3}{2})^{J}) \\
& \geq \mathbb{P}(N_{J_1}^\mu \geq (\frac{3}{2})^{J_1}) \left( \prod_{j=J_1}^J (1-(\frac{2}{3})^j) \right).
\end{align*}
Let us remark that, as $\sum_{j=1}^\infty(\frac{2}{3})^j < \infty$, the infinite product $\prod_{j=1}^\infty (1-(\frac{2}{3})^j)$ converges to a non-zero limit and 
\[ \lim_{J_1 \to + \infty} \prod_{j=J_1}^{+ \infty} (1-(\frac{2}{3})^j)=1.\]
Now, for all $0 < \varepsilon < \frac{1}{2}$, one can choose $J_1$ such that
\[ \prod_{j=J_1}^{+ \infty} (1-(\frac{2}{3})^j) > 1-\varepsilon\]
and, by increasing $\mu$ if necessary, we can choose to remove the intervals $[0,2^{-J_1}]$ and $[1-2^{-J_1},1]$ from $S_{\text{low},J_1}^\mu$, if necessary and assume
\[ \mathbb{P}(N_{J_1}^\mu \geq (\frac{3}{2})^{J_1}) > 1- \varepsilon.\]
Therefore, as the sequence of events $(\{N_J^\mu \geq 1\})_{J\in \N_{0}}$ is decreasing, we get
\begin{align*}
\mathbb{P}\big(\bigcap_{J \in \N_{0}} (N_J^\mu \geq 1)\big) \geq \mathbb{P}(N_{J_1}^\mu \geq (\frac{3}{2})^{J_1}) \left( \prod_{j=J_1}^\infty (1-(\frac{2}{3})^j) \right) >(1-\varepsilon)^2.
\end{align*}
In total, we showed that, for all $0 < \varepsilon < \frac{1}{2}$,
\[\mathbb{P}\big( \bigcup_{\mu\in \N} \bigcap_{J \in\N_{0} } (N_J^\mu \geq 1)\big) >(1-\varepsilon)^2 \] 
and the conclusion follows immediately.
\end{proof}

\subsubsection{Existence of slow points}

Now, using this iterative procedure, we establish the existence of points satisfying inequality \eqref{thm:main3a}. In this section, it will be convenient to write
\begin{align}\label{split}
 f_{h,j}= \sum_{n \in \Z} f_{h,j}^{[n]}
 \end{align}
where, for all $n \in \Z$,
\[  f_{h,j}^{[n]}=\sum_{k= n2^j}^{(n+1)2^j-1} \xi_{j,k} 2^{-h j} \psi(2^j \cdot -k ) \]
is the partial sum involving the $k$ for which the corresponding dyadic interval \linebreak $[k2^{-j},(k+1)2^{-j}[$ is included in $[n,n+1[$.

\begin{Prop}\label{prop:slow}
Almost surely, there exists $t \in (0,1)$ such that
\[  \limsup_{s \to t} \dfrac{|f_h(s)-f_h(t) |}{\omega_{s}^{(h)}(|t-s|)}< + \infty . \]
\end{Prop}

\begin{proof}
Let us fix $m \in \N_{0}$ such that $h \geq 1/m$.  The iterative procedure
with $m$ 
gives 
that almost surely, $(0,1) \cap S_{\text{low}}^\mu \neq \emptyset$. 
Now let $t \in (0,1) \cap S_{\text{low}}^\mu$. There exists $r>0$ such that $[t-r,t+r] \subset (0,1)$ and so, let us take $s \in [t-r,t+r]$ and $\nu \in \N$ be such that
\[ 2^{-\nu} < |t-s| \leq 2^{-\nu +1}.\]
We are going to estimate from above $|f_h(t)-f_h(s)|$. 

For all $j \leq \nu$, once again, by the Mean Value Theorem, there exists $x$ between $s$ and $t$ such that
\[ |f_{h,j}(t)-f_{h,j}(s)| \leq |s-t| |D f_{h,j} (x)|. \]
To bound $|D f_{h,j} (x)|$, we will use the decomposition \eqref{split}. For $n=0$, if $0 \leq l \leq \lceil j/m \rceil$, $0 \leq k < 2^j$ and $|k_j(t)-k| \leq 2^{ml}$, then necessarily $|\xi_{j,k}| \leq 2^{l} \mu$. Thus, if we set
\[ \Lambda_{j}^0(t) = \{ 0  \leq k < 2^j \, : \, |k_j(t)-k| \leq 1 \} \]
and, for all $1 \leq l \leq \lceil j/m \rceil$
\[ \Lambda_{j}^l(t) = \{ 0  \leq k < 2^j \, : \, 2^{m(l-1)}<|k_j(t)-k| \leq 2^{ml} \}, \]
we have, using the fast decay of $D \psi$,
\begin{align}
|Df_{h,j}^{[0]}(x)| & \leq  2^{(1-h) j} \sum_{l=0}^{\lceil j/m \rceil} \sum_{k \in  \Lambda^l(t)} |\xi_{j,k}|  |D\psi(2^j x -k )| \label{eq:slow}\\
& \leq   C 2^{(1-h) j} \sum_{l=0}^{\lceil j/m \rceil} \sum_{k \in
  \Lambda_{j}^l(t)} 2^l \mu \frac{1}{(3+|2^j x-k|)(1+|2^j x-k|)^4}  .
\end{align}
As, for $l \geq 1$ and $k \in \Lambda_{j}^l(t)$, $|2^j x- k| \geq |k_j(t)-k|-3 > 2^{m(l-1)}-3 $, we obtain
\begin{align*}
|Df_{h,j}^{[0]}(x)| & \leq   C \mu 2^{(1-h) j} \sum_{k=0}^{2^j-1} \frac{1}{(1+|2^j x-k|)^4}.
\end{align*}
Now, there exists\footnote{We can make this constant only dependant of $t$. Note that, as $t$ is random, $C_t$ is a random variable as well.} $C_t>0$, such that, for all $ n\neq 0$ and $n2^j \leq k <(n+1)2^j-1$, we have $|2^j x- k| \geq C_t |n| 2^j$, which gives
\begin{align*}
|Df_{h,j}^{[n]}(x)| & \leq C C_2 2^{(1-h) j} \sum_{k= n2^j}^{(n+1)2^j-1} \frac{1}{(1+|2^j x-k|)^4} \frac{\sqrt{\log(3+j+|k|)}}{C_t |n| 2^j} \\
& \leq C C_2 2^{(1-h) j} \sum_{k= n2^j}^{(n+1)2^j-1} \frac{1}{(1+|2^j x-k|)^4}.
\end{align*}
Let us define the random variable $C_\mu=\max(C_2,\mu)$. By Lemma \ref{unibound}, we obtain, for all $j \leq \nu$,
\begin{align*}
|f_{h,j}(t)-f_{h,j}(s)| &\leq C_\mu 2^{(1-h) j}|s-t|
\end{align*}
and thus
\begin{align}\label{eq:slow1}
|\sum_{j \leq \nu}(f_{h,j}(t)-f_{h,j}(s))| 
& \leq C C_\mu 2^{(1-h) \nu}|s-t| \nonumber \\
& \leq C C_{\mu}^* |s-t|^h.
\end{align}

Now, we consider the terms for $j > \nu$ and we will bound separately $|f_{h,j}(t)|$ and $|f_{h,j}(s)|$. For $|f_{h,j}(t)|$, we just have to repeat the same procedure, using the set $\Lambda_j^l(t)$ to estimate $|f_{h,j}^{[0]}(t)|$ and Lemma \ref{boundnormal} for $|f_{h,j}^{[n]}(t)|$ with $n \neq 0$. We then conclude that
\begin{align*}
|\sum_{j > \nu} f_{h,j}(t)| & \leq C \sum_{j > \nu} C_\mu  2^{-h j} \\
& \leq C C_\mu  2^{-h \nu} \\
& \leq C C_\mu |t-s|^h \\
\end{align*}
For $|f_{h,j}(s)|$, the strategy remains the same: nothing changes to bound $|f_{h,j}^{[n]}(s)|$ with $n \neq 0$ while, for $n=0$, if $l$ is the greatest integer such that $|s-t| \geq 2^{ml} 2^{-j}$, the construction insures that, for all $1 \leq l' \leq \lceil j/m \rceil$ and $k \in \Lambda_j^{l'}(s)$,
\[ | \xi_{j,k} | \leq 2^l 2^{l'} \mu\]
and we get
\begin{align*}
 |f^{[0]}_{h,j}(s)|  \leq  C 2^l \mu 2^{-h j}  \sum_{k=0}^{2^j-1} \frac{1}{(1+|2^j s-k|)^4}.
\end{align*}
But as $|t-s| \leq 2^{-\nu +1}$, we have $l \leq \frac{1}{m}(j+1-\nu)$ and thus
\begin{align*}
2^l \mu 2^{-h j} \leq 2^\frac{1}{m} \mu 2^{(\frac{1}{m}-h)(j-\nu)} 2^{-h\nu}.
\end{align*}
It gives, as we took $\frac{1}{m}< h$, combined with Lemma \ref{unibound},
\begin{align}\label{eq:slow2}
|\sum_{j > \nu} f_{h,j}(s)| \leq C C_\mu|t-s|^h.
\end{align}
Finally, combining \eqref{eq:slow1} with \eqref{eq:slow2} allows to obtain 
\begin{align*}
|f_h(t)-f_h(s)| \leq C C_\mu |t-s|^h
\end{align*}
as desired. 
\end{proof}

\section{Asymptotic behavior of wavelet coefficients} \label{sec:wavcoef}

In this section we study the asymptotic behavior of the
wavelet coefficients of the Gaussian wavelet
series. It will allow to get irregularity properties for the random
wavelets series $f_h$, i.e. to prove that the three corrections obtained in the previous
section characterize exactly three possible pointwise behaviors. 
Let us start by recalling the following result from \cite[Lemma A.27]{ayache}. 

\begin{Lemma}\label{lem:lowerbound1}\cite{ayache}
Almost surely, for every $t \in \R$, one has
$$
\limsup_{j \to + \infty} |\xi_{j, k_{j}(t)} |
\geq 2^{-3/2} \sqrt{\pi}. 
$$
\end{Lemma}

Before stating the next lemma, let us recall that if $\xi \sim \mathcal{N}(0,1)$, then one has
\begin{equation} \label{eq:ajout}
\lim_{x \to + \infty} \frac{\PP \big( |\xi|> x\big)}{(2\pi^{-1})^{1/2}x^{-1} e^{-x^2 /2}} =  1 \, . 
\end{equation}
The  following result follows the lines of \cite{ayacheesserkleyntssens}.

\begin{Lemma}\label{lem:lowerbound2}
Almost surely, for almost every $t \in \R$, one has
\begin{equation}\label{eq:lowerbound2}
\limsup_{j \to + \infty }  \dfrac{|\xi_{j, k_{j}(t)}|
}{\sqrt{\log j}} >0 \, .
\end{equation}
\end{Lemma}

\begin{proof}
By Fubini theorem, it is enough to prove that for every
$t \in \R$, \eqref{eq:lowerbound2} holds almost surely. Let us fix $t
\in \R$. For every $m \in \N$, we consider the event
$$
A_{m} (t)= \left\{ \max_{2^{m} \leq j < 2^{m+1}} |\xi_{2^{j},
 k_{j}(t)}| \geq
  \sqrt{m\log 2}\right\}.
$$
Using the independance of the random
variables  $\xi_{j,k}$, $(j,k) \in \N \times\Z$, we have
\begin{eqnarray*}
  \PP\left( A_{m}(t)\right)&=& 1 -  \prod_{2^{m} \leq j < 2^{m+1}} \PP  \left( |\xi_{2^{j},
  k_{j}(t)}| < \sqrt{m\log 2} \right)
  \\
  &=&1 -  \Big(1 - \PP  \big( |\xi| > \sqrt{m\log 2}\,  \big) \Big)^{2^{m}}
  \end{eqnarray*}
where $\xi \sim \mathcal{N}(0,1)$. Using \eqref{eq:ajout} and the fact
that $\log(1-x) \leq -x$ if $x \in (0,1)$, we obtain  for $m$
large enough
\begin{eqnarray*}
  \PP\left(A_{m} (t)\right)
  \geq  1 - \Big(1 - C \dfrac{2^{-\frac{m}{2}}}{\sqrt{m\log 2}} \Big)^{2^{m}}
  & \geq & 1 - \exp\left(- C  \dfrac{2^{\frac{m}{2}}}{\sqrt{m\log 2}}\right)
  \end{eqnarray*}
where $C= \frac{1}{2}\sqrt{\frac{2}{\pi}}$. Hence, it follows that 
\[
\sum_{m=0}^{+ \infty} \PP\left( A_{m}(t)\right)= + \infty \, 
\] 
and since that the events $A_{m}(t)$, $m  \in \N$, are  independents,
the Borel-Cantelli lemma implies that 
\begin{equation*}
\PP \left( \bigcap_{M \in \N} \bigcup_{m \geq M}A_{m}(t) \right) = 1 \, .
\end{equation*}
It gives that  almost surely, for infinitely many  $m \in \N$, there is
$j \in \{ 2^{m}, \dots, 2^{m+1}-1\}$ such that
$$
|\xi_{2^{j}, k_{j}(t)}| \geq \sqrt{m\log 2}\geq\sqrt{\log
j - \log 2}.
$$
The conclusions follows.
\end{proof}

The last lemma we  need is proved in \cite{ayacheesserkleyntssens}. Note that in this
paper, the authors works in a more general context. Indeed, it is only
required that the sequence of standard Gaussian random variables
satisfy the following condition: there is $N \in \N$ such that for
every $(j_{1},k_{1}), (j_{2},k_{2}) \in \N \times \Z$ satisfying
$$
\left(\frac{k_{1}-N}{2^{j_{1}}}, \frac{k_{1}+N}{2^{j_{1}}}\right) \cap
\left(\frac{k_{2}-N}{2^{j_{2}}}, \frac{k_{2}+N}{2^{j_{2}}}\right) =
\emptyset \, ,
$$
the 
random variables $\xi_{j_{1},k_{1}} $ and $\xi_{j_{2},k_{2}}$ are
independant.

\begin{Lemma}\cite{ayacheesserkleyntssens}\label{lem:lowerbound3}
Almost surely, for every non-empty open interval $I$ of $\R$, there is $t\in I$ such that 
\begin{equation*}
\limsup_{j \to + \infty }\left\{ \frac{ |\xi_{j,k_{j}(t)}| }{ \sqrt{j}} \right\} >0 \, .
\end{equation*}
\end{Lemma}



\section{Proof of the main result} \label{sec:proof}

Putting  together  Propositions \ref{prop:rapid},
\ref{prop:ordinary}, \ref{prop:slow} and Lemmata
\ref{lem:lowerbound1}, \ref{lem:lowerbound2} and \ref{lem:lowerbound3}, we can summarize the 
results obtained in the previous sections as follows.

\begin{Cor}\label{cor:main} Almost surely, 
\begin{itemize}
\item there exists $t \in (0,1)$ such that
$$
\limsup_{s \to t} \dfrac{|f_h(s)-f_h(t) |}{\omega^{(h)}_{r}(|t-s|)}< + \infty 
\quad \text{ and }\quad
\limsup_{j \to + \infty} \dfrac{|c_{j, \lfloor2^{j}t\rfloor}|}{\omega_{r}^{(h)}(2^{-j})}>0 ,
$$
\item for almost every $t\in (0,1)$, 
$$
\limsup_{s \to t} \dfrac{|f_h(s)-f_h(t) |}{\omega^{(h)}_{o}(|t-s|)}< + \infty 
\quad \text{ and }\quad
\limsup_{j \to + \infty} \dfrac{|c_{j, \lfloor2^{j}t\rfloor}|}{\omega_{o}^{(h)}(2^{-j})} >0,
$$
\item there exists $t \in (0,1)$  such that
$$
\limsup_{s \to t} \dfrac{|f_h(s)-f_h(t) |}{\omega^{(h)}_{s}(|t-s|)}< + \infty 
\quad \text{ and }\quad
\limsup_{j \to + \infty} \dfrac{|c_{j, \lfloor2^{j}t\rfloor}|}{\omega_{s}^{(h)}(2^{-j})}>0 .
$$
\end{itemize}
\end{Cor}

We are now ready to prove our main result.

\begin{proof}[Proof of Theorem \ref{thm:main}]
Using Corollary \ref{cor:main}, it suffices to prove the second part of
the result. 
Let us consider $t \in (0,1)$ for which there is $v \in \{r,o,s \}$
such that 
\begin{equation}\label{eq:proofmain}
\limsup_{s \to t} \dfrac{|f_h(s)-f_h(t) |}{\omega_{v}^{(h)}(|t-s|)}< + \infty 
\quad \text{ and }\quad
\limsup_{j \to + \infty} \dfrac{|c_{j, \lfloor2^{j}t\rfloor}|}{\omega_{v}^{(h)}(2^{-j})} >0.
\end{equation}
Let us fix a modulus of continuity $\omega$ such
that $\omega=o(\omega_{v}^{(h)})$.  Assume by contradiction that
\begin{equation}\label{eq:proofmain2}
 \limsup_{s \to t} \dfrac{|f_h(s)-f_h(t) |}{\omega(|s-t|)}< + \infty .
\end{equation}
 Given an arbitrary fixed $j
\in \N$, we set $k= \lfloor 2^{j}t\rfloor$. The first vanishing moment of the
wavelet allows to write 
\begin{eqnarray}\label{eq:1}
|c_{j,k}| 
& = & \left| 2^{j}\int_{\R} \psi(2^{j}x-k) 
\big(f_h(x) - f_h(t)\big) dx\right|\nonumber\\
& \leq & 2^{j}\int_{B(t,2^{-j})} |\psi(2^{j}x-k) |
\big|f_h(x) - f_h(t)\big| dx \nonumber\\
&&+ 2^{j} \sum_{l=0}^{j-1} \int_{B_{l}} |\psi(2^{j}x-k) |
\big|f_h(x) - f_h(t)\big| dx \nonumber\\
&& +2^{j} \int_{\R \setminus B(t,1)} |\psi(2^{j}x-k) |
\big|f_h(x) - f_h(t)\big| dx
\end{eqnarray}
where for every $l \in \{0, \dots, j-1\}$, $B_{l}$ denotes the set
$B(t,2^{-l})\setminus B(t,2^{-l-1})$ .
The first term can be controlled by using the modulus of continuity
$\omega$ in the following way
\begin{eqnarray}\label{eq:2}
2^{j}\int_{B(t,2^{-j})} |\psi(2^{j}x-k) |
\big|f_h(x) - f_h(t)\big| dx
& \leq & C \omega(2^{-j})
\end{eqnarray}
for some positive constant $C$ by assumption \eqref{eq:proofmain2}. 
In order to deal with the second term, we use in addition the fast decay of the
wavelet to get the existence of a natural number $N$ and a constant $C$ such that
\begin{eqnarray*}
2^{j} \sum_{l=0}^{j-1} \int_{B_{l}} |\psi(2^{j}x-k) |
\big|f_h(x) - f_h(t)\big| dx
& \leq & C \sum_{l=0}^{j-1} \omega(2^{-l}) \int_{B_{l}} \frac{2^{j}}{(1+
         |2^{j}x-k|)^{2N}} dx.
\end{eqnarray*}
Notice then that if $x \in B_{l}$, then 
$$
1+|2^{j}x-k| \geq 1+ 2^{j}|x-t| - |2^{j}t-k|\geq 2^{j-l-1}.
$$
It implies that
\begin{eqnarray}\label{eq:3}
2^{j} \sum_{l=0}^{j-1} \int_{B_{l}} |\psi(2^{j}x-k) |
\big|f_h(x) - f_h(t)\big| dx
& \leq & C   \sum_{l=0}^{j-1}\omega(2^{-l})2^{-N(j-l-1)} \int_{B_{l}}
         \frac{2^{j}}{(1+|2^{j}x-k|)^{N}} dx \nonumber\\
& \leq & C   \sum_{l=0}^{j-1}\omega(2^{-l})2^{-N(j-l-1)} \int_{\R}
         \frac{1}{(1+|y|)^{N}} dy \nonumber\\
& = & C'  2^{-Nj} \sum_{l=0}^{j-1}\omega(2^{-l}) 2^{Nl} \nonumber\\
& \leq & C'' \omega(2^{-j})
\end{eqnarray}
for some constants $C',C''$ and using \eqref{modreg}. It remains to bound the last term. Again, the fast decay of
the wavelet together with the boundedness of the random wavelets
series 
$f_{h}$  leads to 
\begin{eqnarray}\label{eq:4}
2^{j} \int_{\R \setminus B(t,1)} |\psi(2^{j}x-k) |
\big|f_h(x) - f_h(t)\big| dx 
& \leq & 2 \|f_{h}\|_{\infty} \int_{\R \setminus B(t,1)}
         \frac{2^{j}}{(1+|2^{j}x-k|)^{2N}} dx \nonumber\\
& \leq & 2 \|f_{h}\|_{\infty} 2^{-Nj}\int_{\R }
         \frac{2^{j}}{(1+|y|)^{N}} dy\nonumber\\
& \leq & C 2^{-Nj}
\end{eqnarray}
for some  positive constant $C$, since $|2^{j}x-k| \geq 2^{j}$ if
$|t-x|\geq 1$. Putting \eqref{eq:1}, \eqref{eq:2}, \eqref{eq:3} and \eqref{eq:4}
together, we finally obtain
\begin{eqnarray*}
|c_{j,k}|  \leq C \big( \omega(2^{-j})+ 2^{-Nj})
\end{eqnarray*}
for some positive constant $C$, so that
$$
\limsup_{j \to + \infty}\frac{|c_{j,k}|}{\omega(2^{-j})} < + \infty
$$
if $N$ is chosen large enough. This contradicts the second part of \eqref{eq:proofmain} since $\omega
= o (\omega_{v}^{(h)})$. 
\end{proof}

\begin{Rmk}\label{rem:suppcompact}
As mentionned earlier, let us notice that if the wavelet $\psi$ is compactly supported, and
if $t$ is a rapid, ordinary or slow point, then
$\omega_{v}$ gives the exact pointwise behavior of
$f$ at $t$, meaning that
$$
0 < \limsup_{s \to t} \dfrac{|f(s)-f(t) |}{\omega_{v}^{(h)}(|t-s|)}< + \infty 
$$
where $v=r$, $v=o$ or $v=s$ respectively.
Indeed, in this case, one has directly
\begin{eqnarray*}
|c_{j,k}| & \leq &   2^{j}\int_{B(t, R2^{-j})} |\psi(2^{j}x-k) |
\big|f_h(x) - f_h(t)\big| dx\\
& \leq & C \sup_{x:|x-t|< R 2^{-j}}|f_h(x) - f_h(t)|
\end{eqnarray*}
where $R$ can be computed via the support of the wavelet and $C$ is a
positive constant.  
\end{Rmk}

The last Remak applies in particular to the fractional Brownian motion
thanks to the  representation \eqref{eq:mbf}.  Indeed, as $R$ is a
smooth process, it does not modify the pointwise regularity and
irregularity properites. 

\begin{Cor}
Almost surely, the fractional Brownian motion satisfies the following property for every non-empty interval $I$ of $\R$:
\begin{itemize}
\item  almost every $t\in I$ is ordinary:
\[
 0<\limsup_{s \to t} \dfrac{|f_h(s)-f_h(t) |}{|t-s|^{h} \sqrt{\log \log |t-s|^{-1}}}< + \infty, 
\]
\item  there exists $t\in I$ which is fast: 
\[
 0<\limsup_{s \to t} \dfrac{|f_h(s)-f_h(t) |}{|t-s|^{h} \sqrt{ \log |t-s|^{-1}}}< + \infty, 
\]
\item  there exists $t\in I$ which is slow: 
\[
 0<\limsup_{s \to t} \dfrac{|f_h(s)-f_h(t) |}{|t-s|^{h} }< + \infty.
\]
\end{itemize}
\end{Cor}

\section{Extension to the multifractal case}\label{sec:multi}

Let us now consider a multifractal version of the previously
introduced Gaussian wavelets series. This is in the spirit of the
multifractional Brownian motion introduced in \cite{plv95,bjr97} by
replacing the constant Hurt exponent $h$ in the stochastic integral
defining this  process by a continuous function $H(\cdot)$. In this paper, we adopt the strategy of \cite{bbcI00} which takes advantage of wavelets series expansion. It consists in substituting the exponent $h$ at level $(j,k)$ in \eqref{eq:RWS} by $H(k2^{-j})$. In \cite{ayachebertrand}, the authors showed that, under some H\"olderian regularity assumptions for the function $H$, both generalizations are equivalent in the sense that the multifractional Brownian Motion $B_{H(\cdot)}$ and the Gaussian wavelets series associated to $H(\cdot)$ only differs by a smooth process, similarly to \eqref{eq:mbf}.

Here, we just need a weaker regularity condition to obtain a generalized version of Theorem \ref{thm:main}. Namely, we consider a compact set $K \subseteq (0,1)$ and a function $H \, : \, \R \to K$,  for which there exits a constant $C_H>0$ such that 
\begin{equation}\label{eq:regH}
|H(x)-H(y)| \leq \frac{C_H}{|\log |x-y| |} 
\end{equation}
for all $x,y$ with $|x-y| < 1$. Of course, such a function $H$ is necessarily continuous and any Hölder-continuous function satisfies \eqref{eq:regH}.

This function allows to  define the multifractal random wavelet series
\begin{equation}\label{eq:RWSmulti}
 f_H = \sum_{j \in \N} \sum_{k \in \Z} \xi_{j,k} 2^{-H(k2^{-j})j} \psi(2^j \cdot -k ). \end{equation}
Our goal is to show that, the function $t \mapsto H(t)$ maps any $t
\in \R$ to its Hölder exponent and that, even in this generalized
setting, slow, ordinary and rapid points can still be
highlighted. Note that, in \cite{dlvm98}, it is proved that, if $H$ is
the function ``H\"older exponent'' of a continuous function then there
exists a sequence $(P_j)_{j \in \N_{0}}$ of polynomials such that
\begin{equation}\label{eq:hmeyer}
 \left\{
\begin{array}{l}
H(t) = \liminf_{j \to + \infty} P_j(t) \\[2ex] 
\ \| D P_j \|_\infty \leq j  , \quad \forall j \in \N_{0} \\ 
\end{array} 
\right. 
\end{equation}
In the situation of deterministic wavelet series with coefficients
$(2^{-jH(k2^{-j})})_{j,k}$, conditions \eqref{eq:hmeyer} are also
sufficient to recover the irregularity from the function $H$
\cite{Clausel:11}. Because of condition \eqref{eq:regH}, our function $H$ is not as general, but if a function $H$ checks conditions \eqref{eq:hmeyer} and if we assume the existence of a constant $C>0$ such that, for all $t \in \R$ and $j \in \N_{0}$,
\[ |H(t)-P_j(t)| \leq C j^{-1} \]
then, \eqref{eq:regH} is satisfied.

\begin{Thm}\label{thm:main2}
Let $I$ denote any non-empty  interval of $\R$.
Almost surely, the random wavelets series defined in \eqref{eq:RWSmulti} satisfies the following property:
\begin{itemize}
\item  For almost every $t\in I$, 
\begin{equation}
 \limsup_{s \to t} \dfrac{|f_H(s)-f_H(t)
   |}{\omega_{o}^{(H(t))}(|s-t|)}< + \infty 
\end{equation}
and if $\omega$ is a modulus of continuity such that $\omega =
o(\omega_{o}^{(H(t))})$, then
\begin{equation}
 \limsup_{s \to t} \dfrac{|f_H(s)-f_H(t) |}{\omega(|s-t|)}= + \infty ,
\end{equation}
\item There exists $t \in I$  such that
\begin{equation}
 \limsup_{s \to t} \dfrac{|f_H(s)-f_H(t) |}{\omega_{r}^{(H(t))}(|t-s|)}< + \infty ,
\end{equation}
and if $\omega$ is a modulus of continuity such that $\omega =
o(\omega_{r}^{(H(t))})$, then
\begin{equation}
 \limsup_{s \to t} \dfrac{|f_H(s)-f_H(t) |}{\omega(|s-t|)}= + \infty ,
\end{equation}
\item There exists $t \in I$  such that
\begin{equation}
\limsup_{s \to t} \dfrac{|f_H(s)-f_H(t) |}{\omega_{s}^{(H(t))}}< + \infty 
\end{equation}
and if $\omega$ is a modulus of continuity such that $\omega =
o(\omega_{s}^{(H(t))})$, then
\begin{equation}
 \limsup_{s \to t} \dfrac{|f_H(s)-f_H(t) |}{\omega(|s-t|)}= + \infty .
\end{equation}
\end{itemize}
\end{Thm}


\begin{proof}
The part concerning the irregularity can easily be adapted from the
constant case $H(t)=h$. It suffices to notice that the three superior
limits appearing in Corollary \ref{cor:main} and that concerns the wavelet
coefficients are still stricly postive, using \eqref{eq:regH}. 

The idea and techniques for the regularity are similar to the ones of Propositions \ref{prop:rapid}, \ref{prop:ordinary} and \ref{prop:slow} except that, in addition to ``control'' the randoms coefficients with the help of the fast decay of the wavelets, we also need to ``control'' the variation of the function $2^{H(\cdot)j}$, with $j \in \N$. Again, we will do so using the fast decay of the wavelets.

Let us list how to modify the proofs of Propositions \ref{prop:rapid}, \ref{prop:ordinary} and \ref{prop:slow} in order to find the expected inequalities

\bigskip

\textbf{Rapid points:}

\bigskip

To bound $|Df_{H,j}(x)|$ in \eqref{eq:demorapid1}, the idea consists to split the sum over $|k| \leq 2^{j+1}$ in two and provides additional information for the sum over $|k| > 2^{j+1}$:

\begin{enumerate}[(i)]
\item If $k$ is such that $|k| < 2^{j+1}$ and $|2^jt-k| \leq 2^{j/2}$ then \linebreak $|H(t)-H(k2^{-j})| \leq 2C_H j^{-1}$ and thus
\begin{align*}
2^{(1-H(k2^{-j}))j} |\xi_{j,k}| |D \psi (2^jx-k)| &\leq 2^{2C_H} 2^{(1-H(t)) j}  |\xi_{j,k}| |D\psi (2^jx-k)| \\
& \leq C 2^{2C_H} 2^{(1-H(t)) j} \frac{\sqrt{\log(3+j+|k|)}}{(1+|2^jx-k|)^4}.
\end{align*}

\item If $k$ is such that $|k| < 2^{j+1}$ and $|2^jt-k| >2^{j/2}$, we have
\[ |2^jx -k | > 2^{j/2}-2 \]
and thus, if $\beta > 2 (\sup K - \inf K)$, we get, using the fast decay of $D \psi$,
\begin{align*}
2^{(1-H(k2^{-j}))j} |\xi_{j,k}| |D \psi (2^jx-k)| &\leq C   \frac{2^{(1-H(t))j} 2^{(\sup K - \inf K)j} |\xi_{j,k}|}{(1+|2^jx-k|)^4 (2+|2^jx-k|)^\beta} \\
& \leq C  2^{(1-H(t))j} \frac{\sqrt{\log(3+j+|k|)}}{(1+|2^jx-k|)^4}.
\end{align*}

\item For $|k| > 2^{j+1}$ of course $|2^jx-k| > 2^j$ and one can use the same trick to get
\[ 2^{(1-H(k2^{-j}))j} |\xi_{j,k}| |D \psi (2^jx-k)| \leq C 2^{2C_H} 2^{(1-H(t)) j} \frac{\sqrt{\log(3+j+|k|)}}{(1+|2^jx-k|)^5}. \]
\end{enumerate}

From this, the bound
\[ | \sum_{j \leq \nu} (f_{H,j}(t)-f_{H,j}(s))|  \leq C C_2 |t-s|^{H(t)} \sqrt{\log|t-s|^{-1}}\]
is obtained.

The same method applied on $|f_j(t)|$ ($j > \nu)$ leads to 
\[ |\sum_{j > \nu} f_{H,j}(t)| \leq C C_2|t-s|^{H(t)} \sqrt{\log|t-s|^{-1}} \]
while, for $|f_{H,j}(s)|$, replacing $t$ by $s$ in the reasoning, we get
\begin{align*}
 |\sum_{j > \nu} f_{H,j}(s)| &\leq C C_2 2^{-H(s) \nu} \sqrt{\nu} \\
 & \leq C C C_2 2^{-H(t) \nu} \sqrt{\nu} 2^{(H(t)-H(s)) \nu} \\
 & \leq C C_2 |t-s|^{H(t)} \sqrt{\log|t-s|^{-1}}\
\end{align*}
as $|H(t)-H(s)| \leq (\nu-1)^{-1}$.

\bigskip

\textbf{Ordinary points}

\bigskip

In this case, to bound $|Df_{H,j}(x)|$ (with $j \leq \nu$), the sum for $k \in \kappa_j^t(j)$ remains untouched as, for all such $k$, we have $|t-k2^{-j}| \leq 2^{-j}(j+1)$ and thus we get $|H(t)-H(2^{-j})| \leq C C_H j^{-1}$.

This time, this is the sum for $k \notin \kappa_j^t(j)$  that we need to split:

\begin{enumerate}[(i)]
\item If $k \in \kappa_j^t(2^{j/2}) \setminus \kappa_j^t(j)$, $|k2^{-j}-t| \leq 2 2^{-j/2}$ and, similarly to the point (i) in the rapid points, we have
\[ 2^{(1-H(k2^{-j}))j} |\xi_{j,k}| |D \psi (2^jx-k)| \leq C 2^{2C_H} 2^{(1-H(t)) j} \frac{\sqrt{\log(3+j+|k-k_j(t)|)}}{(3+|2^jx-k|)(1+|2^jx-k|)^4}. \]

\item If $k \notin \kappa_j^t(2^{j/2})$ then $|2^jx-k| \geq 2^{j/2}-3$ and we proceed just like in the point (ii) for the rapid points to get
\[ 2^{(1-H(k2^{-j}))j} |\xi_{j,k}| |D \psi (2^jx-k)| \leq C 2^{2C_H} 2^{(1-H(t)) j} \frac{\sqrt{\log(3+j+|k-k_j(t)|)}}{(3+|2^jx-k|)(1+|2^jx-k|)^4} \]
again.
\end{enumerate}

To bound $|f_{H,j}(t)|$ (with $j > \nu$) the strategy remains the same while to bound $|f_{H,j}(s)|$, like in the proof of Proposition \ref{prop:rapid} we need to consider separately the case where $j \geq 2^{j-\nu+2}$ and when $j < 2^{j-\nu+2}$.

If $j \geq 2^{j-\nu+2}$, we take into account the three sums:
\begin{enumerate}[(i)]
\item the sum over $k \in \kappa_j^t(2j)$ where $|H(t)-H(2^{-j})| \leq C C_H j^{-1}$ and thus which remains untouched,
\item the sum over $k \in \kappa_j^t(2^{j/2}) \setminus \kappa_j^t(2j)$ where we use $|H(t)-H(2^{-j})| \leq C C_H j^{-1}$ to deal with the exponent,
\item the sum over $k \notin \kappa_j^t(2^{j/2})$ where we deal with the exponent just like in (ii) and (iii) for the rapid points.
\end{enumerate}
In total, we get
\[ |f_{H,j}(s)| \leq C C_t 2^{-j H(t)} \sqrt{\log(j)} \]

Now if $2^{j-\nu+2}<j$ we consider:
\begin{enumerate}[(i)]
\item the sum over $\kappa_j^s(2^{j/2})$ for which we deal similarly to the sum over $\kappa_j^t(2^{j/2}) $ above
\item the sum over $\kappa_j^s(2^{j-\nu +2}) \setminus \kappa_j^s(2^{j/2})$ where the exponent is treated like in (ii) and (iii) for the rapid points.
\item the sum for $k \notin \kappa_j^s(2^{j-\nu +2})$ where, again, we treat the exponent like in (ii) and (iii) for the rapid points.
\end{enumerate}
In total, in this case, we get
\[ |f_{H,j}(s)| \leq C C_t 2^{-j H(s)} \sqrt{j-\nu}. \]
The conclusion follows from the same arguments as for the rapid
points. 
\bigskip

\textbf{Slow points}

\bigskip

First, we note that the natural  $m$ in Kahane procedure must be chosen such that \linebreak $1/m < \inf K$.

To bound $|D f_{H,j}^{[0]}(x) |$ (with $j \leq \nu$), we split the sum \eqref{eq:slow} in two:
\begin{enumerate}[(i)]
\item for $0 \leq l \leq \lfloor j/2m \rfloor$, then for all $k \in
  \Lambda_j^l(t)$, we have $|x-k2^{-j}| \leq C 2^{-j/2}$ and thus one
  uses $|H(x)-H(k2^{-j})| \leq C C_H j^{-1}$ to deal with the exponent.
\item for $\lfloor j/2m \rfloor < l \leq \lceil j/m \rceil$, then for all $k \in \Lambda_j^l(t)$, $|2^jx-k| \geq 2^{j/2}-3$ and we treat the exponent like for the rapid points in (ii).
\end{enumerate}

To bound $|D f_{H,j}^{[n]}(x) |$ with $n \neq 0$, we use $|2^jx-k| \geq C_t |n| 2^j$ to treat the exponent like in the rapid points in (iii).

With this, we obtain
\begin{align*}
|\sum_{j \leq \nu}(f_{H,j}(t)-f_{H,j}(s))| & \leq  C C^* |s-t|^\alpha.
\end{align*}

We can use the same method to have
\[ |\sum_{j > \nu} f_{H,j}(t)| \leq C C_\mu  2^{-H(t) \nu} \]
and, using the same arguments as in the end of Proposition \ref{prop:slow}, with the fact that $1/m < \inf K \leq H(s)$, we have
\[ |\sum_{j > \nu} f_{H,j}(t)| \leq C C_\mu 2^{-H(s) \nu} \leq C C_\mu  2^{-H(t) \nu}  \]
just as in the conclusion for the rapid points.

\end{proof}

\section{Genericity of the non-existence of slow points}\label{sec:genericity}

The aim of this section is to prove that the results obtained in the
previous section are spectific to these Gaussian random wavelet
series. On this purpose, let us define the Fréchet space
\[ C^{\nearrow h}:=\bigcap_{\alpha < h} C^{\alpha}([0,1]),\]
where for every $\alpha \in (0,h)$, $C^{\alpha}([0,1])$ denote the H\"older
space of order $\alpha$. This space and its topology can be
equivalently defined using sequence of wavelet coefficients as
follows: If we consider $\alpha > 0$,  $\alpha \notin  \mathbb{N}$, then for any $f
\in C^{\alpha }([0,1])$, one has 
 \begin{equation}\label{eq:holderwavecaract}
  \sup_{j \in \N} \sup_{k \in \{0, \dots, 2^j-1\}} 2^{\alpha j} |c_{j,k}|
  < + \infty 
\end{equation}
 where $(c_{j,k})_{j \in \N, k \in \{0, \dots,
  2^j-1\}}$ denotes the sequence of wavelet coefficients of $f$ on $[0,1]$. It
allows to identify (algebraically and topologically) 
 the H\"older space $C^{\alpha } ([0,1]
) $ with the  space of complex sequences $  (c_{j,k})_{j \in \N, k
  \in \{0, \dots, 2^j-1\}}$ satisfying   \eqref{eq:holderwavecaract}, 
isee \cite{Meyer:95}. When $\alpha  \in  \mathbb{N}$,  we will also
denote by $C^{\alpha } ([0,1]) $ the space of functions satisfying
the condition \eqref{eq:holderwavecaract}.

We will prove that in $C^{\nearrow h}$, ``most of'' the functions do
not present a slow pointwise behavior as the one exhibited for Gaussian
wavelet series, i.e.  does not belong to the space $C^{h}(t)$ formed
by the
functions $f$ such that
$$
\limsup_{s \to t} \frac{|f(s)-f(t)|}{|s-t|^{h}}< + \infty.
$$
 We will use in this purpose two notions of
genericity: the prevalence and the Baire category points of view. 


The notion of prevalence supplies an extension of the notion of ``almost everywhere'' (for the  Lebesgue measure)  in infinite dimensional spaces. In
a metric infinite dimensional vector space,  no measure  is both
$\sigma$-finite  and translation invariant. However, one can consider
a natural extension of  the notion of ``almost everywhere'' which is
translation invariant, see  \cite{Christensen74,HSY92}. 

\begin{Def}\label{def:prev}
Let $E$ be a complete metric vector space. A Borel set $A\subset E$ is  {Haar-null} if there exists a compactly supported probability measure $\mu$  
such that
\begin{equation}\label{trans}
\forall x\in E,\quad \mu(x+A)=0. 
\end{equation} 
  If this property holds, the measure $\mu$ is said to be {transverse}
  to $A$. A subset of E is called {Haar-null} if it is contained in a Haar-null Borel set. The complement of a Haar-null set is called a {prevalent set}.
\end{Def}



In order to prove that a set is Haar-null in a functional space $E$, one can often use for transverse measure the
law of a stochastic process, see
\cite{ClauselNicolay2010,EsserJaffard} for some applications of this method. 
If 
$\mathcal{P}$ is a property that can be satisfied by points of  $E$,   
one can prove that  
$\mathcal{P}$ holds  only on a Haar-null set by exhibiting  a stochastic process $X$ whose sample paths  lies in a compact subset of $E$
and such that  for all $f \in E$ almost surely the property
$\mathcal{P}$ does not hold for $X +f$.

In our setting, the stochastic process that will be used is a random
wavelet series. The following result  allows to get that the sample
paths of this
series are almost surely in a compact set of $C^{\nearrow h}$. Let us
first describe this subset.

\begin{Lemma}\label{lem:prevalence}
Let $h>0$ and let $(\alpha_{j})_{j \in \N}$ be a non-decreasing sequence of $(0,h)$
with tends to $h$. The subset
$$
K= \left\{ f \in C^{\nearrow h}:
  \max_{k \in \{0, \dots, 2^{j}-1\}}|c_{j,k} |\leq 2^{-\alpha_{j}j}\,\, \forall j \in \N\right\}
$$
is compact in $C^{\nearrow h}$, where $(c_{j,k})_{j \in \N, k \in \{0, \dots,
  2^j-1\}}$ denotes the sequence of wavelet coefficients of $f$. 
\end{Lemma}

\begin{proof}
Clearly, $K$ is closed in $C^{\nearrow h}$. Since this last space is 
a Fréchet space, it
suffices now to prove that $K$ is totally bounded. Let us fix $\varepsilon >0$ and
$\alpha<h$. Then there exists $J \in \N$ such that $2^{-\alpha_{j}j}<
\varepsilon 2^{-\alpha j}$ for every $j \geq J$, which implies that
$$
\sup_{j \geq J, k \in \{0, \dots, 2^{j}-1\}}2^{\alpha j}|c_{j,k}| < \varepsilon
$$
for the sequence of wavelet coefficients $(c_{j,k})_{j \in \N, k \in
  \{0, \dots, 2^{j}-1\}}$ of any $f\in
K $. Moreover, since the product
$$
P= \prod_{j=0}^{J-1}\prod_{k=0}^{2^{j}-1}[-2^{-\alpha_{j,k}j}, 2^{\alpha_{j,k}j}]
$$
where we have set $\alpha_{j,k}=\alpha_{j}$ for every $k \in \{0,
  \dots, 2^{j}-1\}$, 
is compact, one can find a finite number of sequences $c^{1}, \dots, c^{L}$ with support included in $\{(j,k):
j \leq J-1, k \in \{0, \dots, 2^{j}-1 \}\}$ such that
$$
P \subseteq \bigcup_{l=1}^{L}
\prod_{j=0}^{J-1}\prod_{k=0}^{2^{j}-1}\big(c^{l}_{j,k}-\delta, c^{l}_{j,k}+\delta\big),
$$
where $\delta  <2^{-\alpha J} \varepsilon$. For every $l$, we define $f^{l} = \sum_{j
  =0}^{J-1}\sum_{k=0}^{2^{j}-1}c^{l}_{j,k} \psi(2^{j}\cdot -k)$. In order to conclude, let
us show that
$$
K \subseteq \bigcup_{l=1}^{L} \left\{ f \in C^{\nearrow h}: 
\|f -f^{l}\|_{\alpha}< \varepsilon
\right\}. 
$$
If $f \in K$, its truncated sequence $(c_{j,k})_{j \leq J-1, k \in
  \{0, \dots, 2^{j}-1 \}}$ of wavelet coefficients belongs to $P$. Hence, for a $l \in
\{1, \dots, L\}$, one has
$$
 2^{\alpha j} |c_{j,k}-c^{l}_{j,k}| \leq \delta 2^{\alpha j} < \varepsilon
$$
if $j \leq J-1$, and
$$
 2^{\alpha j} |c_{j,k}-c^{l}_{j,k}| =  2^{\alpha j} |c_{j,k}|< \varepsilon
$$ 
if $j \geq J$. 
\end{proof}

Before stating our result, we need to recall the following wavelet
``almost characterization'' of the pointwise H\"older regularity. It relies  on alternative quantities, namely the wavelet
leaders. In order to define them, we need to introduce the notation  $c_{\lambda}$
to denote the wavelet coefficient $c_{j,k}$, where $\lambda$ is the dyadic interval 
 \[
 \lambda = \lambda(j,k) = \left[ \frac{k}{2^j}, \frac{k+1}{2^j}\right)
 \]
Then, if  $3 \lambda$ denote the interval with the same center as
$\lambda$ but three times larger, the wavelet leader $d_{\lambda}$ is
defined by  
\begin{equation}\label{eq:leader}
d_{\lambda} = \sup_{ \lambda' \subset 3 \lambda} |c_{\lambda'}|.
\end{equation}
Note that this supremum is finite as soon as $f$ is locally
bounded. In \cite{Jaffard:04b}, the author proved that if $f \in
C^{h}(t)$ for some $h>0$, then there exists a constant $C>0$ such that 
\begin{equation}\label{eq:pointwiseleader}
d_{\lambda(j,k_{j}(t)) } \leq C 2^{-\alpha j}.
\end{equation}
This inequality will allow us to construct wavelet series which do not
belong to $C^{h}(t)$ for every $t \in [0,1]$. 


\begin{Prop}
Let $h>0$. The set of functions $f$
such that $f \notin C^{h}(t) $ for every $t \in [0,1]$ is 
prevalent  in $C^{\nearrow h}$.
\end{Prop}

\begin{proof}
Let us consider a sequence $(j_{n})_{n \in \N}$ 
 satisfying  $j_{n+1}>j_{n}+
\lfloor \log_{2}j_{n}^{2}\rfloor+1$  and let us
set 
  $\alpha_{n} =h-\frac{1}{\sqrt{j_{n}}}$
for every $n \in \N$. 
Let us define the random wavelets series
$$
f = \sum_{n \in \N}\sum_{j=j_{n}}^{j_{n+1}-1}\sum_{k=0}^{2^{j}-1} 2^{-\alpha_{n} j} \varepsilon_{j,k}\psi_{j,k}
$$
where $(\varepsilon_{j,k})_{j \in \N, k \in \{0, \dots, 2^{j}-1\}}$ is
a sequence of independent $\mathcal{U}([-1,1])$ random variables. Clearly, for every $\alpha < h$, one has
$$
2^{-\alpha_{n}j} \leq  2^{-\alpha j}
$$
if $j \in \{j_{n}, \dots, j_{n+1}-1\}$ for $n$ large enough. Using the
characterization given in 
\eqref{eq:holderwavecaract}, it
follows that $f \in C^{\alpha}([0,1])$. 
Moreover, from Lemma \ref{lem:prevalence}, we know that the process
$f$ takes its values in a compact subset of $C^{\nearrow h}$. 

Let us now show that almost surely $f \notin C^{h}(t)$ for every $t \in
[0,1]$. Let us fix $M \in \N$. For every $n \in \N$, if one consider
the subintervals of scale $j'=j_{n}+
\lfloor \log_{2}j_{n}^{2}\rfloor+1$  in the supremum appearing in the
definition  \eqref{eq:leader} of the 
wavelet leaders, one has 
\begin{eqnarray*}
\PP\left(\inf_{k \in \{0, \dots, 2^{j_{n}}-1\}}d_{j_{n},k}\leq
  M 2^{-h j_{n}}\right)
& \leq & \sum_{k=0}^{2^{j_{n}}-1} \PP\left(\sup_{\lambda' \subseteq \lambda_{j_{n},k}} |2^{-\alpha_{n}j'}\varepsilon_{\lambda'}|\leq
  M 2^{-h  j_{n}}\right)\\
& \leq & 2^{j_{n}} \left( 2^{\alpha_n j'}  M  2^{-h j_{n}}\right)^{2^{j'-j_{n}}}\\
& \leq &  2^{j_{n}+ j_{n}^{2}h}  j_{n}^{2 hj_{n}^{2}}2^{- j_{n}^{5/2}}
         M^{j_{n}^{2}}. \\
\end{eqnarray*}
The
Borel-Cantelli Lemma implies that almost surely, one has 
$$
d_{j_{n},k} > M 2^{-hj_{n}}
$$
for every $n$ large enough and every $k \in \{0, \dots,
2^{j_{n}}-1\}$. Since $M \in \N$ is arbitrary,
\eqref{eq:pointwiseleader} gives the announced
result.
In order to conclude, it suffices to prove that the previous result is still valid if we replace the wavelet
series $f$ by
$$
\tilde{f}= f+ {g}
$$
for a function ${g} \in C^{\nearrow h}$. 
In this case, the wavelet coefficients $ 2^{-\alpha_{n} j}$ of $f$ are
replaced by 
$$
2^{-\alpha_{n} j} \varepsilon_{j,k}+  c_{j,k} = 2^{-\alpha_{n}j}
(\varepsilon_{j,k} + 2^{\alpha_{n}j}c_{j,k}).
$$
It implies that the random variables defining the wavelets series are
still independant but
no more centered since they are shifted by a deterministic
quantity. Clearly, the probabilities
computed before can only become smaller, hence the Borel-Cantelli
lemma still holds. 
\end{proof}

To end the paper, we show that the same result holds true if one replaces
the notion of prevalence by the genericity in the sense supplied by the Baire
category theorem. Let us recall that a subset $A$ of a Baire space $X$
is of first category (or meager) if it is included in a countable
union of closed sets of $X$ with empty interior. The complement of a
set of first category is Baire-residual; it contains a countable
union of dense open sets of $X$. 

\begin{Prop}
Let $h>0$. The set of functions $f$
such that $f \notin C^{h}(t) $ for every $t \in  [0,1]$ is 
Baire-residual  in $C^{\nearrow h}$.
\end{Prop}

\begin{proof}
Let us consider the non-decreasing sequence
$(\alpha_{j})_{j \in \N}$ of $(0,h)$
with converges to $h$ defined by $\alpha_{j}= h - \frac{1}{\sqrt{j}}$. For every $J \in \N$, the set
$\mathcal{C}_{J}$ is formed by the functions
$ f \in C^{\nearrow h}$ whose sequence of wavelet coefficients
$(c_{j,k})_{j \in \N , k \in \{0, \dots, 2^{j}-1\}}$ satisfies
$$
 2^{\alpha_{j}j}|c_{j,k}| \in \N \setminus\{0\} \quad  \forall j\geq
 J\, ,  \forall k \in \{0, \dots, 2^{j}-1\}.
$$
Finally, we define the open sets $U_{J}$  by
$$
U_{J} =\bigcup_{j \geq J} \left\{ g \in C^{\nearrow h} : \exists f
\in \mathcal{C}_{j} \text{ such that } \|f-g\|_{\alpha_{j}}< \frac{1}{2}\right\}.
$$
Notice that if $g \in U_{J}$ with wavelet coefficients $(e_{j,k})_{j
  \in \N , k \in \{0, \dots, 2^{j}-1\}}$, then there is $f \in
\mathcal{C}_{j_{0}}$ for a  $j_{0} \geq J$,  with wavelet coefficients $(c_{j,k})_{j
  \in \N , k \in \{0, \dots, 2^{j}-1\}}$ such that
$$
|e_{j_{0},k}|\geq  |c_{j_{0},k}| - |e_{j_{0},k}-c_{j_{0},k}| \geq 2^{- \alpha_{j_{0}}j_{0}}
-\frac{1}{2}2^{-\alpha_{j_{0}}j_{0}}=\frac{1}{2}2^{- \alpha_{j_{0}}j_{0}}
$$
for every $k \in \{0, \dots, 2^{j_{0}}-1\}$. Now, assume that  $g$
belongs to the set $\mathcal{R}$ defined by
$$
\mathcal{R} = \bigcap_{J \in \N} U_{J}.
$$
If there exists $t \in [0,1]$ such that $g \in C^{h}(t)$,
\eqref{eq:pointwiseleader} gives the existence of a constant $C>0$
such that
$$
d_{j,k_{j}(t)} \leq C 2^{-hj}.
$$
Since $g \in \mathcal{R}$, one  gets 
$$
\frac{1}{2}2^{- \alpha_{j}j} = \frac{1}{2}2^{- hj + \sqrt{j}} \leq C 2^{-hj}
$$
for infintely many $j$,  which is impossible. Consequently, $g \notin
C^{h}(t)$ for every $t\in [0,1]$.

To conclude, it suffices
to prove that the open sets $U_{J}$ are dense in $C^{\nearrow h} $. Let us fix $J \in \N $, $g \in C^{\nearrow h}$,  $\alpha <h$ and $\varepsilon >0$. Let
$J_{0}\geq J$ be large enough to ensure that both $\alpha < \alpha_{J_{0}}$ and
$2^{(\alpha-\alpha_{J_{0}})J_{0}} < \varepsilon$ are satisfied. We
construct the function $f$ via its sequence of wavelet coefficients  by setting
$$
c_{j,k} = \begin{cases}
e_{j,k} & \text{ if } j < J_{0} \\
2^{-\alpha_{j} j}  [2^{\alpha_{j} j} e_{j,k}]  & \text{ if } j \geq J_{0}
\text{ and }
2^{\alpha_{j} j}| e_{j,k}|\geq 2,\\
2^{-\alpha_{j} j}  & \text{ if } j \geq J_{0}
\text{ and }
2^{\alpha_{j} j}| e_{j,k}|< 2,
\end{cases}
$$
so that $f \in \mathcal{C}_{J_{0}}\subseteq U_{J}$ and $|2^{\alpha_{j}j}c_{j,k} - 2^{\alpha_{j}j}e_{j,k}|\leq 1$ for
every $j \geq J_{0}$. Hence one has
$$
2^{\alpha j}|e_{j,k}-c_{j,k}| \leq 2^{(\alpha-\alpha_{j})j} \leq 2^{(\alpha-\alpha_{J_{0}})J_{0}} < \varepsilon
$$
 for every $j \geq J_{0}$, which allows to conclude. 
\end{proof}

\bibliography{biblio}{}
\bibliographystyle{plain}

\end{document}